\documentclass[12pt,reqno]{amsart}
\usepackage[margin=1in]{geometry}
\usepackage{amsmath,amssymb,amsthm,graphicx,amsxtra, setspace}
\usepackage[utf8]{inputenc}
\usepackage{mathrsfs}
\usepackage{hyperref}
\usepackage{upgreek}
\usepackage{mathtools}
\usepackage{xcolor}
\allowdisplaybreaks

\usepackage[pagewise]{lineno}

\usepackage{graphicx,eurosym}
\usepackage{hyperref}
\usepackage{mathtools}

\usepackage[cyr]{aeguill}

\colorlet{darkblue}{blue!50!black}

\hypersetup{
	colorlinks,%
	citecolor=blue,%
	filecolor=red,%
	linkcolor=darkblue,%
	urlcolor=blue,%
	pdfnewwindow=true,%
	pdfstartview={FitH}
}

\colorlet{darkblue}{red!100!black}

\newtheorem{theorem}{Theorem}[section]
\newtheorem{lemma}[theorem]{Lemma}

\newtheorem{definition}[theorem]{Definition}

\newtheorem{remark}[theorem]{Remark}
\newtheorem{hypothesis}[theorem]{Hypothesis}

\DeclareMathOperator*{\esssup}{ess\,sup}

\let\originalleft\left
\let\originalright\right
\renewcommand{\left}{\mathopen{}\mathclose\bgroup\originalleft}
\renewcommand{\right}{\aftergroup\egroup\originalright}

\theoremstyle{definition}

\def\1{\mathcal{O}}
\def\t{t\wedge\tau_R^\e}
\def\s{s\wedge\tau_R^\e}
\def\T{T\wedge\tau_R^\e}
\def\tt{t\wedge\tau_{R,\e}^m}
\def\TT{T\wedge\tau_{R,\e}^m}

\def\vi{\widetilde}

\def\d{\mathrm{d}}

\def\I{\mathrm{I}}
\def\B{\mathrm{B}}

\def\A{\mathrm{A}}

\def\W{\mathrm{W}}
\def\R{\mathbb{R}}
\def\E{\mathbb{E}}

\def\H{\mathbb{H}}
\def\V{\mathbb{V}}
\def\e{\epsilon}
\def\2{\mathcal{E}}
\def\L{\mathrm{L}}
\def\u{\boldsymbol{u}}
\def\v{\boldsymbol{v}}

\def\n{\boldsymbol{n}}
\def\f{\boldsymbol{f}}
\def\C{\mathrm{C}}
\def\X{\mathbf{X}}
\def\Y{\mathbf{Y}} 
\def\F{\mathrm{F}}
\def\P{\mathbb{P}}
\def\N{\mathbb{N}}
\def\x{\boldsymbol{x}}
\def\y{\boldsymbol{y}}

\def\Z{\mathrm{Z}}

\def\U{\mathbb{U}}

\def\x{\boldsymbol{x}}
\def\a{\boldsymbol{a}}
\def\z{\boldsymbol{z}}
\def\b{\boldsymbol{b}}

\newcommand{\Addresses}{{% additional braces for segregating \footnotesize
		\footnote{
			\noindent \textsuperscript{1,2}Department of Mathematics, Indian Institute of Technology Roorkee-IIT Roorkee,
			Haridwar Highway, Roorkee, Uttarakhand 247667, INDIA.\par\nopagebreak
			\noindent  \textit{e-mail:} \texttt{Manil T. Mohan: maniltmohan@ma.iitr.ac.in, maniltmohan@gmail.com.}
			
			\textit{e-mail:} \texttt{Ankit Kumar: akumar14@mt.iitr.ac.in.}
			
			\noindent \textsuperscript{*}Corresponding author.
			
			\textit{Keywords:} Stochastic partial differential equations, locally monotne, small time asymptotics, large deviation princple, Gaussian noise.
			
			Mathematics Subject Classification (2020): Primary 60H15, 60F10; Secondary 76S05, 35R60, 35Q35.

}}}
\begin{document}	%	\linenumbers
	
	\title[Small time asymptotics for a class of SPDE]{Small time asymptotics for a class of stochastic partial differential equations with fully monotone coefficients forced  by multiplicative Gaussian noise
		\Addresses}

	\author[A. Kumar and M. T. Mohan]
	{Ankit Kumar\textsuperscript{1} and Manil T. Mohan\textsuperscript{2*}}

	\maketitle

	\begin{abstract}
	The main goal of this article is to study  the effect of small, highly nonlinear, unbounded drifts (small time large deviation principle (LDP) based on exponential equivalence arguments)  for a  class of stochastic partial differential equations (SPDEs) with fully monotone coefficients driven by multiplicative Gaussian noise. The small time LDP obtained in this paper is applicable for various quasi-linear and semilinear SPDEs such as  porous medium equations,  Cahn-Hilliard equation, 2D Navier-Stokes equations,  convection-diffusion equation,  2D liquid crystal model,  power law fluids,  Ladyzhenskaya model,  $p$-Laplacian equations, etc., perturbed by multiplicative Gaussian noise.  
	\end{abstract}
	\section{Introduction}\label{Sec1}\setcounter{equation}{0}
	
	In this work, we analyze the small time asymptotics of the following class of stochastic partial differential equations (SPDEs) with fully local monotone coefficients in a  Gelfand triplet $\V\hookrightarrow\H\hookrightarrow\V^*$ driven by a multiplicative Gaussian noise:
	\begin{equation}\label{1.1}
		\left\{
		\begin{aligned}
			\d \Y(t)&=\A(t,\Y(t))\d t+\B(t,\Y(t))\d\W(t), \ \ t\in[0,T],\\
			\Y(0)&=\x\in \H,
		\end{aligned}
		\right.
	\end{equation}
where  $\H$ and $\V$ represent a separable Hilbert space and a reflexive Banach space, respectively such that the continuous  embedding of  $\V\hookrightarrow \H$ is dense. Let $\V^*$ and $\H^*(\cong \H)$ denote  the dual of the spaces $\V$ and $\H$, respectively. The norms of $\H,\V$ and $\V^*$ are denoted by $\|\cdot\|_\H,\|\cdot\|_\V$ and $\|\cdot\|_{\V^*}$, respectively. Let $(\cdot,\cdot)$ represent the inner product in the Hilbert space $\H$ and $\langle \cdot,\cdot\rangle$ denote the duality paring between $\V$ and $\V^*$. Also, we have $\langle \x,\y\rangle =(\x,\y)$, whenever $\x\in\H$ and $\y\in\V$. Let $\U$ be an another Hilbert space, and $\L_2(\U,\H)$ be the space of all Hilbert-Schimdt operators from $\U\to\H$ with the norm $\|\cdot\|_{\L_2}$ and the inner product $(\cdot,\cdot)_{\L_2}$.  The mappings 
	\begin{align*}
		\A:[0,T]\times\V\to\V^* \ \text{ and } \ \B:[0,T]\times\V\to\L_2(\U,\H),
	\end{align*}
are measurable and $\W(\cdot)$ is a $\U$-valued cylindrical Wiener process on a filtered probability space $(\Omega,\mathscr{F},\{\mathscr{F}_t\}_{t\geq 0},\P)$.	

%Let $(\Omega,\mathscr{F},\P)$ be a complete probability space satisfying the usual properties.
\subsection{Assumptions} We start with a basic definition related to the operator $\A(\cdot,\cdot)$.
\begin{definition}\label{def1}
	An operator $\A$ from $\V$ to $\V^*$ is said to be \textsl{pseudo-monotone} if the following condition holds:  for any sequence $\{\u_n\}_{n\in\N}$ with the weak limit $\u$ in $\V$ and 
	\begin{align}\label{3.2}
		\liminf_{n\to\infty}		\langle \A(\u_n),\u_n-\u\rangle\geq 0,
	\end{align}imply that 
	\begin{align}\label{3.3}
		\limsup_{n\to\infty}	\langle \A(\u_n),\u_n-\v\rangle \leq \langle \A(\u),\u-\v \rangle, \ \text{ for all } \ \v\in\V.
	\end{align}
\end{definition}

\subsubsection{Hypothesis and solvability results} Let first discuss the solvability results for the problem \eqref{1.1}.  The system \eqref{1.1} with locally monotone coefficients is considered in the works \cite{WLMR1,WLMR2,WLMR3}, etc.  The global solvability of the system \eqref{1.1} is proved in \cite{MRSSTZ} under certain assumptions on the coefficients, which are listed below. 
\begin{hypothesis}\label{hyp1} Let $f\in\L^1(0,T;\R^+)$ and $\beta\in(1,\infty)$.
	\begin{itemize}
		\item[(H.1)] (\textsl{Hemicontinuity}). The map $\R\ni \lambda \mapsto \langle \A(t,\u_1+\lambda \u_2),\u\rangle\in\R$ is continuous for any $\u_1,\u_2,\u\in\V$ and for a.e. $t\in[0,T]$.
			\item[(H.2)] (\textsl{Local monotonicity}). There exist non-negative constants $\zeta$ and $C$ such that for any $\u_1,\u_2\in\V$ and a.e. $t\in[0,T]$, 
			\begin{align}\label{1.2}\nonumber
				2\langle \A(t,\u_1)-\A(t,\u_2),\u_1-\u_2\rangle &+\|\B(t,\u_1)-\B(t,\u_2)\|_{\L_2}^2\\& \leq \big(f(t)+\rho(\u_1)+\eta(\u_2)\big)\|\u_1-\u_2\|_\H^2, \\ \label{1.3}
				|\rho(\u_1)|+|\eta(\u_1)| &\leq C(1+\|\u_1\|_\V^\beta)(1+\|\u_1\|_\H^\zeta),
			\end{align} where $\rho$ and $\eta$ are two measurable functions from $\V$ to $\R$.
			\item[(H.3)]	\textsl{(Coercivity)}. There exists a positive constant $C$ such that for any $\u\in\V$ and a.e.  $t\in[0,T]$, 
		\begin{align}\label{1.4}
			\langle \A(t,\u),\u\rangle 
			%+\|\B(t,\x)\|_{\L_2}^2+\int_{\Z} \|\gamma(t,\x,z)\|_{\H}^2\lambda(\d z) 
			\leq f(t)(1+\|\u\|_{\H}^2)-L_\A\|\u\|_{\V}^\beta.
		\end{align}
		\item[(H.4)] \textsl{(Growth)}. There exist non-negative constants  $\alpha$ and $C$ such that for any $\u\in\V$ and a.e. $t\in[0,T]$,
		\begin{align}\label{1.5}
			\|\A(t,\u)\|_{\V^*}^{\frac{\beta}{\beta-1}}\leq \big(f(t)+C\|\u\|_{\V}^\beta\big)\big(1+\|\u\|_{\H}^\alpha\big).
		\end{align} 
		\item[(H.5)]For any sequence $\{\u_n\}_{n\in\N}$ and $\u$ in $\V$ with $\|\u_n-\u\|_\H\to0$, as $n\to\infty$, we have 
		\begin{align}\label{1.6}
			\|\B(t,\u_n)-\B(t,\u)\|_{\L_2}\to 0, \  \text{ for a.e. }t\in[0,T]. 
		\end{align}
		Moreover, there exists $g\in \L^1(0,T;\R^+)$  such that for any $\u\in \V$ and a.e. $t\in[0,T]$, 
		\begin{align}\label{1.7}
			\|\B(t,\u)\|_{\L_2}^2\leq g(t)(1+\|\u\|_{\H}^2). 
		\end{align}
	\end{itemize}
\end{hypothesis}
Let us now recall the following definition of the solution.
	\begin{definition}
	Let $\big(\Omega,\mathscr{F},\{\mathscr{F}_t\}_{t\geq 0},\P\big)$	be a stochastic basis and $\x\in\H$. Then, \eqref{1.1} has a \textsl{probabilistically strong solution} if and only if there exists a progressively measurable process $\Y:[0,T]\times\Omega\to\H,\;\P$-a.s., with paths 
	\begin{align*}
		\Y(\cdot,\cdot)\in \C([0,T];\H)\cap \L^\beta(0,T;\V), \ \P\text{-a.s.,  for } \  \beta\in(1,\infty),
	\end{align*}and the following equality holds $\P$-a.s., in $\V^*$, for all $t\in[0,T]$, 
\begin{align*}
	\Y(t)=\x+\int_0^t\A(s,\Y(s))\d s+\int_0^t\B(s,\Y(s))\d \W(s). 
\end{align*}
\end{definition}Let us recall the well-posedness result from \cite{MRSSTZ}.
\begin{theorem}[Theorem 2.6,  \cite{MRSSTZ}]\label{thrm1}
Assume that Hypothesis \ref{hyp1} (H.1)-(H.6) hold and the embedding $\V\subset \H$ is compact.  Then, for any initial data $\x\in\H$, there exists a 	\textsl{probabilistically strong solution} to the system \eqref{1.1}. Furthermore, for any $p\geq 2$, the following energy estimate holds:
\begin{align}\label{1.8}
	\E\bigg[\sup_{t\in[0,T]}\|\Y(t)\|_\H^p\bigg]+\E\bigg[\int_0^T\|\Y(t)\|_\V^\beta\d t\bigg]^\frac{p}{2}<\infty.
\end{align}
\end{theorem}
\subsection{Literature survey} In the past few decades, the large deviation theory received the required attention due to its applications in the areas like risk management, mathematical finance, fluid mechanics, quantum physics and statistical mechanics, etc. (cf.  \cite{ABPD1,ABPDVM,PLC,ADOZ,MTMLDPJMFM,SSSPS,SRSV1}, etc. and  references therein). The small time large deviation principle (LDP) basically focus on the asymptotic behaviors of the tails of a family of probability measures at a given point in the space when the time is very small. More specifically, it observes the limiting behavior of the solution in a time interval $[0,t]$ as $t$ tends to 0. An inspiration for considering such kind of problems arises from the following Varadhan-type small time asymptotics for the diffusion process $\X$:
\begin{align*}
	\lim_{t\to0}2t\log\P\big\{\X(0)\in\mathrm{E},\; \X(t)\in\mathrm{F}\big\}=-\vartheta^2(\mathrm{E},\mathrm{F}),
\end{align*}
where  $\vartheta$ is an appropriate Riemann distance associated with the diffusion.

 The celebrated theory of small time LDP by Varadhan \cite{SRSV2}  considered the small time asymptotics for the finite dimensional diffusion processes. The small time LDP for infinite dimensional diffusion process has been considered in the works \cite{ZQCSZFTSZ,MHKM,SRSV4,TSZ1}, etc. Many authors contributed in the direction of small time LDP for different types of SPDEs. The most significant results in this direction are the small time LDP for   stochastic 2D Navier-Stokes equations in \cite{TGXTSZ},  stochastic 3D tamed Navier-Stokes equations in \cite{MRTZ1},   stochastic 2D non-Newtonian fluids in \cite{HLCS}, stochastic quasi-geostrophic equations in the sub-critical case in \cite{WLMRXCZ}, 3D stochastic primitive equations in \cite{ZDRZ2}, scalar stochastic conservation laws in \cite{SLWLYX}, stochastic convective Brinkman-Forchheimer equations in \cite{MTMLDPJMFM}, a class of SPDEs with locally monotone coefficients, in \cite{SLWLYX},  stochastic Ladyzhenskaya-Smagorinsky equation in \cite{MTMLDS}, etc. Most of the works discussed above deal with the small time asymptotics  for semilinear type SPDEs except \cite{SLWLYX}.

% The variational framework has been used intensively to analyze SPDEs with the coefficients satisfying classical monotonicity and coercivity conditions. The first phenomenal works in this direction was discussed in \cite{EP2,NVKBLR}, where the authors used the monotonicity method to obtain the well-posedness of the solutions for a class of SPDEs. Later, this theory has been applied to extend the existing result  for more general class of SPDEs with more generalized coercivity and local monotonicity conditons on the coefficients in the works \cite{WLMR1,WLMR2,WLMR3}, etc and  references therein.

% In last few years, several authors studied various properties for different types of SPDEs with monotone and locally monotone coefficients and a good number of literature is available for instance see \cite{WL2,WLCTJZ,JRXZ,JXJZ}, etc and  references therein.
%, random attractors established in \cite{KKMTM,BG1,BG2,BGWLMR2,BGWLAS}, Wong-Zakai approximation and support theorem proved in \cite{TMRZ}, ultra-exponential convergence and existence of optimal controls are discussed in \cite{ECTJF}.

\subsection{Novelties, difficulties and approaches}  The motivation for this paper comes form \cite{MRSSTZ}, where the authors established the well-posedness for a class of SPDEs with fully local monotone coefficients driven by a multiplicative Gaussian noise. Very recently, in the work \cite{AKMTM4}, the authors established  the well-posedness results   for a class of SPDEs with fully local monotone coefficients driven by L\'evy noise.  The Wentzell-Freidlin type LDP results for the same class of SPDEs driven by Gaussian and L\'evy noises have been discussed in \cite{TPSS,AKMTM5}, respectively. %As per our knowledge SPDEs with the locally monotone coefficients are well explored (cf. \cite{ZBWLJZ,SLWLYX,WL4,WLMR1,WLMR2,WLMR3,CPMR}, etc., and references therein), but the system with fully local monotone coefficients are not much investigated in the past, since the existence and uniqueness of \textsf{probabilistically strong solutions} were not established earlier. 
Our  aim of this work is to establish a small time LDP for the system \eqref{1.1} which covers a large class of physical models. 

The idea of the proof of our main result has been borrowed from \cite{TSZ1} (cf. \cite{SLWLYX}), which depends on the exponential equivalence arguments of measures. This exponential equivalence arguments is a very powerful tool and has been used by several authors to establish the small time LDP for different types of SPDEs (see \cite{YCHJGLLF,ZDRZ2,HLCS,WLMRXCZ,MRTZ1,TGXTSZ,ZDRZ3}, etc. and references therein). In this method, one has to consider the following  zero drift stochastic evolution equation
\begin{align}\label{12}
	\X^\e(t)=\x+\sqrt{\e}\int_0^t\B(\e s,\X^\e(s))\d \W(s),
\end{align}
 with the same initial data (see \eqref{2.3} below), where the small noise and small  time asymptotic problems are equivalent. Since the small noise LDP for the solution $\X^\e(\cdot)$ of the SPDE \eqref{12}     holds (Theorem 12.9, \cite{DaZ}), then our aim is to prove that the law of $\Y^\e(\cdot)$ and $\X^\e(\cdot)$ are exponentially equivalent (see Theorem 4.2.13, \cite{ADOZ} and \eqref{2.4} below).

In the  variational framework given in \cite{MRSSTZ}, the authors  considered the Gelfand triplet $\V\hookrightarrow\H\hookrightarrow\V^*$, where $\V$ is a reflexive Banach space such that the embedding $\V\hookrightarrow \H$ is compact. In this work, we consider $\V$ as a separable 2-smooth Banach space, as we are using the following Burkholder-Davis-Gundy type inequality  for $p>2$ (see subsection \ref{ss1.4} for the definition of $\gamma$)
\begin{align}\label{112}
	\E\left[\sup_{t\in[0,T]}\left\|\int_0^t\B(\e s,\X^{\e}(s))\d\W(s)\right\|_{\V}^p\right]\leq C\sqrt{p}\E\left[\left(\int_0^T\|\B(\e s,\X^{\e}(s))\|_{\gamma}^2\d s\right)^{\frac{p}{2}}\right],
\end{align}
with the sharp constant $p^{1/2}$  from Theorem 1.1, \cite{JS} to obtain the required estimate in $\V$-norm (see \eqref{2.21} below and we refer to \cite{SLWLYX} for similar techniques).  
%In the work \cite{SLWLYX}, the authors overcome to this difficulty and use the concept of 2-smooth Banach spaces and to obtain the $\V$-norm estimates, they used the Burkholder-Davis-Gundy type inequality established in the paper \cite{JS} (cf. \cite{JZZBWL}) for stochastic integrals in the 2-smooth Banach space with the sharp constant $p^{1/2}$ which plays a  crucial role in the proof. 
Even though It\^o's formula in $2$-smooth Banach spaces are known (cf. Theorem A.1, \cite{ZBSP}, Theorem 2.1, \cite{ZB}, etc.), we are not using it to obtain the estimate \eqref{2.21} as the twice Fr\'echet differentiabilty of the norm in $\V$ is not known.  The  examples of 2-smooth Banach spaces include $\mathbb{L}^p$ spaces with $p\geq2$ and the Sobolev spaces $\mathbb{W}^{s,p}$ with $s\geq 1$ and $p\geq2$ (see \cite{JNMR} for more examples).  Thus, our main result  is applicable to different types of SPDEs such as 2D Navier-Stokes equations, fast-diffusion equations,  porous media equations, $p$-Laplacian equations, Allen-Cahn equations, Burgers equations, 2D Boussinesq system, 3D Leray-$\alpha$ model, 2D Boussinesq model for the Benard convection, 2D magneto-hydrodynamic equations, 2D magnetic B\'enard equations, shell models of turbulence (Sabra, Goy, Dyadic), power law fluids, the Ladyzhenskaya model, 3D tamed Navier-Stokes equations, the Kuramoto-Sivashinsy equations, and many more. % To the best of our knowledge this is the first result which deals with the small time LDP for a class of SPDEs with fully local monotone coefficients, which covers several interesting examples mentioned above. 

The major differences with the work \cite{SLWLYX} are 
\begin{enumerate}
	\item [(1)] The local monotonicity coefficients of  $\A(\cdot,\cdot)$  in \eqref{1.1} is allowed to  depend on both the variables $\u_1$ and $\u_2$ (cf. \eqref{1.2} and \eqref{1.3}). This covers many interesting examples like Cahn-Hilliard equation (subsection \ref{ss3.1}), Quasilinear SPDEs (subsection \ref{ss3.3}), 2D Liquid crystal model (subsection \ref{ss3.5}), etc. 
	\item [(2)] The function $f$ appearing in \eqref{1.2} and \eqref{1.5} is assumed to be in $\mathrm{L}^1(0,T;\R^{+})$, whereas in \cite{SLWLYX} it is assumed to be a constant. 
	\item [(3)] The functions $g_1,h$ and $L_B$ appearing in \eqref{1.09}, \eqref{1.10} and \eqref{1.11} below are assumed to be in $\mathrm{L}^{\infty}(0,T;\R^+)$, while they are assumed to be constant in \cite{SLWLYX}. 
\end{enumerate}
We also remark that the results of this work may be extended to 2-smooth UMD (unconditional martingale differences) Banach spaces   (see discussion after Theorem 4.9, \cite{JNMR1}). Even though the maximal inequalities are known (cf. Theorem 1.1, \cite{IY}),  the difficulty in considering UMD Banach spaces is that an estimate similar to \eqref{112} with the sharp constant $p^{1/2}$ is not known in the literature. 
\subsection{Main result} \label{ss1.4}
Our main aim in this work  is to analyze the small time LDP for the solutions of the system \eqref{1.1} with fully local monotone coefficients.  We need to find an estimate for the maximal inequality for  the stochastic integral $\int_0^t\B(s,\cdot(s))\d\W(s)$ in the Banach space $\V$ in an appropriate way.  Therefore, we assume that the space $\V$ is a 2-smooth Banach space. Let $\gamma(\U,\V)$ be the space of all $\gamma$-radonifying operators from $\U$ to $\V$. Let us recall that  if $\mathrm{T}\in\gamma(\U,\V)$ , then the series 
\begin{align*}
	\sum_{j\in\N}\gamma_j \mathrm{T}(\mathbf{e}_j) \text{ converges in } \L^2(\vi{\Omega};\V),
\end{align*}for any sequence $\{\gamma_j\}_{j\in\N}$ of independent Gaussian real-valued random variables defined on a probability space $(\vi{\Omega},\vi{\mathscr{F}},\vi{\P})$ and any orthonormal family $\{\mathbf{e}_j\}_{j\in\N}$ of the Hilbert space $\U$. Then, the  space $\gamma(\U,\V)$ is endowed with the norm 
\begin{align*}
	\|\mathrm{T}\|_{\gamma(\U,\V)} :=\bigg\{\vi{\E}\bigg\|	\sum_{j\in\N }\gamma_j \mathrm{T}(\mathbf{e}_j)\bigg\|_\V^2\bigg\}^\frac{1}{2},
\end{align*}which does not  depend on the choice of $\{\gamma_j\}_{j\in\N}, \{\mathbf{e}_j\}_{j\in\N}$, and form a  Banach space. For convenience we denote the norm of the space $\gamma(\U,\V)$ by $\|\cdot\|_\gamma$ instead of $\|\cdot\|_{\gamma(\U,\V)}$.
% For the orthonormal family $\{\mathbf{e}_j\}_{j\in\N}$ and an independent standard Gaussian sequence $\{\gamma_j\}_{j\in\N}$ on the probability space $(\vi{\Omega},\vi{\mathscr{F}},\vi{\P})$, using H\"older's inequality, stochastic Fubini's theorem, Gaussian formula (see Appendix A.1, \cite{DN}) and Minkowski's inequality, we have (see pp. 16, \cite{MTMLDS}) \begin{align*}	\|\B(t,\u)\|_{\gamma(\U,\V)}^2\leq \int_\R \frac{|\xi|^2}{\sqrt{2\pi}}e^{-\frac{\xi^2}{2}}\d \xi k(t)\big(1+\|\u\|_\V^2\big)\end{align*}where $k\in\L^1(0,T;\R^+)$.

\begin{remark}
	It should be noted that if  $\V$ a separable Hilbert space, then $\V$ is 2-smooth, which implies $\gamma(\U,\V)$ consists of all Hilbert-Schmidt operators from $\U$ to $\V$, and $\|\cdot\|_\gamma=\|\cdot\|_{\L_2}$ (see Example 2.8, \cite{XCZ}). Also, examples of 2-smooth Banach spaces include  every Hilbert space, $\mathbb{L}^p$  with $p\geq 2$ and Sobolev spaces $\mathbb{W}^{s,p}$ with $p\geq 2$ and $s\geq1$ (cf. \cite{JNMR}).
\end{remark}
In order to prove small time LDP for the solutions of the system \eqref{1.1}, we need the following  additional assumptions on the noise coefficients:
\begin{hypothesis}\label{hyp2}The coefficients $\A$ and $\B$ satisfy the following assumptions:
	\begin{itemize}
		\item[(H.2)$^*$] (\textsl{Local monotonicity}). There exist non-negative constants $\zeta$,  $C$  and $\xi>0$ such that for any $\u_1,\u_2\in\V$ and a.e. $t\in[0,T]$, 
		\begin{align}\label{1.9}
			2\langle \A(t,\u_1)-\A(t,\u_2),\u_1-\u_2\rangle\nonumber &+\|\B(t,\u_1)-\B(t,\u_2)\|_{\L_2}^2+\xi\|\u_1-\u_2\|_\V^\beta\\& \leq \big(f(t)+\rho(\u_1)+\eta(\u_2)\big)\|\u_1-\u_2\|_\H^2,
		\end{align} where $\rho$ and $\eta$ are the two measurable functions from $\V$ to $\R$ defined in \eqref{1.3}.
%		\item[(H.3)$^*$]	\textsf{(Coercivity)}. There exists a positive constant $C$ such that for any $\u\in\V$ and a.e.  $t\in[0,T]$, 
%	\begin{align}\label{1.009}
%		\langle \A(t,\u),\u\rangle 
%		%+\|\B(t,\x)\|_{\L_2}^2+\int_{\Z} \|\gamma(t,\x,z)\|_{\H}^2\lambda(\d z) 
%		\leq f_1(t)(1+\|\u\|_{\H}^2)-L_\A\|\u\|_{\V}^\beta.
%	\end{align}
%	\item[(H.4)$^*$] \textsf{(Growth)}. There exist non-negative constants  $\alpha$ and $C$ such that for any $\u\in\V$ and a.e. $t\in[0,T]$,
%\begin{align}\label{1.05}
%	\|\A(t,\u)\|_{\V^*}^{\frac{\beta}{\beta-1}}\leq \big(f_1(t)+C\|\u\|_{\V}^\beta\big)\big(1+\|\u\|_{\H}^\alpha\big).
%\end{align} 
	\item[(H.6)] (\textsl{Growth}). There exist functions $g_1,h\in\L^\infty(0,T;\R^+)$ such that for any $\u\in\V$ and a.e. $t\in[0,T]$, 
	\begin{align}\label{1.09}
		\|\B(t,\u)\|_{\L_2}^2 \leq g_1(t)\big(1+\|\u\|_\H^2\big),
	\end{align}and
\begin{align}\label{1.10}
\|\B(t,\u)\|_\gamma^2\leq h(t)\big(1+\|\u\|_\V^2\big).
\end{align}	
Moreover, there exists a function $L_\B\in \L^\infty(0,T;\R^+)$ such that for any $\u_1,\u_2\in\V$ and a.e. $t\in[0,T]$,
		\begin{align}\label{1.11}
			\|\B(t,\u_1)-\B(t,\u_2)\|_{\L_2}^2\leq L_\B(t)\|\u_1-\u_2\|_\H^2.
		\end{align}
	\end{itemize}
\end{hypothesis}
\begin{remark}
	Note that Hypothesis \ref{hyp2} (H.2)$^*$ and (H.6) are stronger than   Hypothesis \ref{hyp1} (H.2)  and (H.5), respectively. 
\end{remark}
\subsubsection{Large deviation principle}
Let $(\mathcal{E},\nu)$ be a Polish space. We are given a family of probability measures $\{{\varrho_\e}\}_{\e>0}$ on $\mathcal{E}$ and a lower semicontinuous function $\I:\mathcal{E}\to[0,+\infty]$, not identically equals to infinity and such that its level sets, 
$$ K_M:=\{\ell\in\mathcal{E}:\I(\ell)\leq M\},\ \ M>0,$$ are compact for any $M\in[0,+\infty]$. The family $\{\varrho_\e\}_{\e>0}$ is said to satisfy the LDP or to have a large deviation property with respect to the rate function $\I(\cdot)$ if 
\begin{itemize}
	\item[(1)] for any closed set $\F \subset  \mathcal{E}$, we have
	\begin{align*}
		\limsup_{\e\to0} \e \log \varrho_\e (\F) \leq -\inf_{\ell\in\F}\I(\ell),
	\end{align*}
	\item[(2)] for any open set $\mathrm{G} \subset   \mathcal{E}$, we have
	\begin{align*}
		\liminf_{\e\to0} \e \log \varrho_\e (\mathrm{G}) \geq -\inf_{\ell\in\mathrm{G}}\I(\ell).
	\end{align*}
\end{itemize}

Define a functional $\I(\ell)$ on $\C([0,T];\H)$ by 
\begin{align}\label{1.12}
	\I(\ell):=\inf_{\ell\in\Lambda_k}\bigg\{\frac{1}{2}\int_0^T\|\dot{\ell}(t)\|_\U^2\d t\bigg\},
\end{align}where 
\begin{align*}
	\Lambda_k=\bigg\{& \ell\in \C([0,T];\H): \ell(\cdot) \text{ is absolutely continuous and such that }\\&\quad k(t)=x+\int_0^t\B(s,k(s))\dot{\ell}(s)\d s,\ t\in[0,T]  \bigg\}.
\end{align*}
For $\e>0$, we aim to study the probabilistic asymptotic behavior for small time process $\{\Y^{\e}(t)\}_{t\geq 0}=\{\Y(\e t)\}_{t\geq 0}$ as $\e\to 0$. Let us now state the  main result of this paper.
\begin{theorem}\label{thrm2}
	Let $\varrho_{\x}^{\e}$ be the law of $\Y^\e(\cdot)$ in $\C([0,T];\H)$. Then  under Hypotheses \ref{hyp1} and \ref{hyp2},  $\varrho_{\x}^{\e}$ satisfies an LDP with the rate function $\I(\cdot)$, that is, 
	\begin{itemize}
		\item[(1)] for any closed set $\F \subset  \C([0,T];\H)$,
		\begin{align*}
			\limsup_{\e\to0} \e \log \varrho_{\x}^{\e} (\F) \leq -\inf_{\ell\in\F}\I(\ell),
		\end{align*}
		\item[(2)] for any open set $\mathrm{O} \subset  \C([0,T];\H)$,
		\begin{align*}
			\liminf_{\e\to0} \e \log \varrho_{\x}^{\e} (\mathrm{O}) \geq -\inf_{\ell\in\mathrm{O}}\I(\ell).
		\end{align*}
	\end{itemize}
\end{theorem}
We prove the above theorem in the subsequent section.

\subsection{Organization of the paper}The rest of the article is organized as follows: 
%In the last part of the section \ref{Sec1} we provide the basic definitions related to LDP, additional Hypothesis \ref{hyp2} used to obtain small time LDP for the system  \eqref{1.1} and state our main Theorem \ref{thrm2}. 
A proof of Theorem \ref{thrm2} is discussed in section \ref{Sec2}. In order to prove Theorem \ref{thrm2}, we provide a sufficient condition \eqref{2.4} (see Theorem 4.2.13, \cite{ADOZ}) on the exponentially equivalency of the probability measures. The proof of the sufficient condition \eqref{2.4} has been divided into a series of Lemmas \ref{lemma1},  \ref{lemma2},  \ref{lemma3},  \ref{lemma4},  \ref{lemma5},   and  \ref{lemma6}, which lead to the proof of our main Theorem \ref{thrm2}. In the final section \ref{Sec3}, we apply Theorem \ref{thrm2} to some important  models like  Cahn-Hilliard equation (subsection \ref{ss3.1}),  2D Navier-Stokes equations (subsection \ref{ss3.2}), Quasilinear SPDEs (subsection \ref{ss3.3}), convection-diffusion equation (subsection \ref{ss3.4}) and 2D Liquid crystal model (subsection \ref{ss3.5}).

\section{Proof of Theorem \ref{thrm2}}\label{Sec2}\setcounter{equation}{0}
In this section, we establish the small time asymptotics for the solutions of the system \eqref{1.1} by establishing an LDP for small time, which is based on the exponential equivalence arguments. Under Hypotheses \ref{hyp1} and \ref{hyp2}, by Theorem \ref{thrm1}, the system \eqref{1.1} has a unique strong solution  with paths in $\C([0,T];\H)\cap \L^\beta(0,T;\V)$, for $\beta\in(1,\infty)$, $\mathbb{P}$-a.s., and the following holds:
\begin{align}\label{2.1}
		\Y(t)=\x+\int_0^t\A(s,\Y(s))\d s+\int_0^t\B(s,\Y(s))\d \W(s),
\end{align}in $\V^*$, for all $t\in[0,T],\ \P$-a.s. Our aim is to prove an LDP for the family $\{\Y(\e \cdot):\e\in(0,1]\}$ for the solution of the system \eqref{1.1}. Using the scaling property of Wiener process, it is easy to see  that the law of $\Y^{\e}(\cdot)=\Y(\e \cdot)$ is the law of the solution of the following stochastic evolution equation:
 \begin{align}\label{2.2}
 	\Y^\e(t)=\x+\e\int_0^t\A(\e s,\Y^\e(s))\d s+\sqrt{\e}\int_0^t\B(\e s,\Y^\e(s))\d \W(s),
 \end{align}in $\V^*$, for all $t\in[0,T],\ \P$-a.s.
\begin{proof}[Proof of  Theorem \ref{thrm2}]
	Let us denote the solution $\X^\e(\cdot)$ for the following stochastic evolution equation:
	\begin{align}\label{2.3}
		\X^\e(t)=\x+\sqrt{\e}\int_0^t\B(\e s,\X^\e(s))\d \W(s),
	\end{align}in $\V^*$, for all $t\in[0,T]$, and we express the law by $\vi{\varrho}_{\x}^\e$ of $\X^\e(\cdot)$ on the space $\C([0,T];\H)$. By Theorem 12.9, \cite{DaZ}, we note that $\X^\e(\cdot)$ satisfies the LDP with the rate function $\I(\cdot)$. Our aim is to show that two families of the measures $\varrho_{\x}^\e$ and  $\vi{\varrho}_{\x}^\e$ are exponentially equivalent, that is, for any $\varpi>0$, 
\begin{align}\label{2.4}
	\lim_{\e\to0}\e\log\P\bigg\{\sup_{t\in[0,T]}\|\Y^\e(t)-\X^\e(t)\|_\H^2>\varpi \bigg\}=-\infty.
\end{align}In view of Theorem 4.2.13, \cite{ADOZ}, if we are able to show that \eqref{2.4} holds then we are done. Since Theorem 4.2.13, \cite{ADOZ} says that if we have two exponentially equivalent families of probability measures and LDP holds for one of them, then it holds for the other probability measure also.
\end{proof}

Now, our main focus is to establish \eqref{2.4} only. We divide the proof into a series of Lemmas. The following lemma provides an estimate of the probability that the solution of \eqref{2.2}  leaves an energy ball.

\begin{lemma}\label{lemma1}
	For any $p\geq 2$, we have 
	\begin{align}\label{2.5}
		\lim_{R\to\infty}\sup_{\e\in(0,1]}\e\log\P\bigg\{\bigg(|\Y^\e|_\H^\V(T)\bigg)^p>R\bigg\}=-\infty,
	\end{align}where 
\begin{align*}
	\bigg(|\Y^\e|_\H^\V(T)\bigg)^p=\sup_{t\in[0,T]}\|\Y^\e(t)\|_\H^p+\frac{pL_\A\e}{2} \int_0^T\|\Y^\e(t)\|_\H^{p-2}\|\Y^\e(t)\|_\V^\beta\d t. 
\end{align*}
\end{lemma}
\begin{proof}
	Applying infinite dimensional It\^o's formula to the process $\|\Y^\e(\cdot)\|_\H^p$ (see Theorem 1.2, \cite{IGDS}), we find
	\begin{align}\label{2.6}\nonumber
	&	\|\Y^\e(t)\|_\H^p \\&\nonumber=\|\x\|_\H^p+\frac{p\e}{2}\int_0^t\|\Y^\e(s)\|_\H^{p-2}\big[2\langle\A(\e s,\Y^\e(s)),\Y^\e(s)\rangle +\|\B(\e s,\Y^\e(s))\|_{\L_2}^2 \big]\d s\\&\nonumber\quad+\frac{p(p-2)\e}{2}\int_0^t \|\Y^\e(s)\|_\H^{p-4}\|\B(\e s,\Y^\e(s))\circ\Y^\e(s)\|_\U^2\d s\\&\nonumber\quad + p\sqrt{\e}\int_0^t \|\Y^\e(s)\|_\H^{p-2}\big(\B(\e s,\Y^\e(s))\d \W(s),\Y^\e(s)\big)
	 \\&\nonumber\leq \|\x\|_\H^p-\frac{pL_\A\e}{2}\int_0^t \|\Y^\e(s)\|_\H^{p-2}\|\Y^\e(s)\|_\V^\beta\d s+\frac{p\e}{2}\int_0^tf(\e s)\big(1+\|\Y^\e(s)\|_\H^2\big)\|\Y^\e(s)\|_\H^{p-2}\d s\\&\nonumber\quad +\frac{p(p-1)\e}{2}\int_0^tg_1(\e s)\big(1+\|\Y^\e(s)\|_\H^2 \big)\|\Y^\e(s)\|_\H^{p-2}\d s\\&\quad + p\sqrt{\e}\int_0^t \|\Y^\e(s)\|_\H^{p-2}\big(\B(\e s,\Y^\e(s))\d \W(s),\Y^\e(s)\big),
	\end{align}
$\P$-a.s., for all $t\in[0,T]$, where we have used Hypotheses \ref{hyp1} (H.2) and \ref{hyp2} (H.6). Using H\"older's and Young's inequalities in the third and fourth terms of the right hand side of the  inequality \eqref{2.6}, we deduce 
	\begin{align}\label{2.7}\nonumber
	&	\|\Y^\e(t)\|_\H^p +\frac{pL_\A\e}{2}\int_0^t \|\Y^\e(s)\|_\H^{p-2}\|\Y^\e(s)\|_\V^\beta\d s\\&\nonumber\leq \|\x\|_\H^p +C_p  \int_0^{ T}\big(f(s)+g_1(s)\big)\d s+C_p\e \int_0^t\big(f(\e s)+g_1(\e s)\big)\|\Y^\e(s)\|_\H^p\d s\\&\quad + p\sqrt{\e}\int_0^t \|\Y^\e(s)\|_\H^{p-2}\big(\B(\e s,\Y^\e(s))\d \W(s),\Y^\e(s)\big),
		\end{align}$\P$-a.s., for all $t\in[0,T]$. Taking supremum from $0$ to $T$ and for any $q\geq 2$, we find
	\begin{align}\label{2.8}\nonumber
		&	\bigg\{\E\bigg[\bigg(|\Y^\e|_\H^\V(T)\bigg)^{pq}\bigg]\bigg\}^{\frac{1}{q}} \\&\nonumber \leq \|\x\|_\H^p+C_p \int_0^T\big(f(s)+g_1(s)\big)\d s+C_p\e \bigg\{\E\bigg[\int_0^T
		\big(f(\e s)+g_1(\e s)\big)\big(|\Y^\e|_\H^\V(s)\big)^p\d s\bigg]^q\bigg\}^\frac{1}{q}\\&\quad + C_p\sqrt{\e} \bigg\{\E\bigg[\sup_{t\in[0,T]}\bigg|\int_0^t\|\Y^\e(s)\|_\H^{p-2}\big(\B(\e s,\Y^\e(s))\d \W(s),\Y^\e(s)\big)\bigg|^q\bigg]\bigg\}^{\frac{1}{q}}. 
		\end{align}To estimate the final term of the right hand side of the  inequality \eqref{2.8}, we use the following result (cf. \cite{MTBMY,BDA}). There exists a positive constant $C$ such that, for any $q\geq 2$ and for any continuous martingale $\Z=\{[\Z_t]\}_{t\geq 0}$ with $\Z_0=0$, we have 
	\begin{align}\label{2.9}
		\big\{\E\big[|\Z_t^*|^q\big]\big\}^{\frac{1}{q}} \leq Cq^{\frac{1}{2}}\big\{\E\langle \Z\rangle_t^{\frac{q}{2}}\big\}^{\frac{1}{q}},	
	\end{align}where $\Z_t^*=\sup\limits_{s\in[0,t]}|\Z_s|$ and $\langle \Z\rangle_{\cdot}$ is the quadratic variation process. Using \eqref{2.9}, Hypothesis \ref{hyp2} (H.6),  H\"older's, Young's  and Minkowski's inequalities, we arrive at 
\begin{align}\label{2.10}\nonumber
	&\sqrt{\e}\bigg\{\E\bigg[\sup_{t\in[0,T]}\bigg|\int_0^t \|\Y^\e(s)\|_\H^{p-2}\big(\B(\e s,\Y^\e(s))\d \W(s),\Y^\e(s)\big)\bigg|^q\bigg]\bigg\}^{\frac{1}{q}} \\&\nonumber\leq C \sqrt{q\e}\bigg\{\E\bigg[\int_0^T \|\Y^\e(s)\|_\H^{2p-2}\|\B(\e s,\Y^\e(s))\|_{\L_2}^2\d s\bigg]^{\frac{q}{2}}\bigg\}^{\frac{1}{q}}\\&\nonumber\leq C \sqrt{q\e}\bigg\{\E\bigg[\int_0^T \|\Y^\e(s)\|_\H^{2p-2}g_1(\e s)\big(1+\|\Y^\e(s)\|_\H^2\big)\d s\bigg]^{\frac{q}{2}}\bigg\}^{\frac{1}{q}}\\&\nonumber\leq C_p \sqrt{q\e}\bigg\{\E\bigg[\int_0^Tg_1(\e s)\d s+ \int_0^T g_1(\e s)\|\Y^\e(s)\|_\H^{2p}\d s\bigg]^{\frac{q}{2}}\bigg\}^{\frac{1}{q}}\\&\nonumber\leq C_p \sqrt{q}\bigg(\int_0^{T}g_1(s)\d s\bigg)^{\frac{1}{2}}+C_p \sqrt{q\e}\bigg\{\E\bigg[ \int_0^T g_1(\e s)\|\Y^\e(s)\|_\H^{2p}\d s\bigg]^{\frac{q}{2}}\bigg\}^{\frac{1}{q}}\\&\leq C_p \sqrt{q\e T}\bigg(\esssup\limits_{s\in[0,T]}g_1(s)\bigg)^{\frac{1}{2}}+C_p \sqrt{q\e}\bigg\{\int_0^T g_1(\e s)\big\{\E\big[ \|\Y^\e(s)\|_\H^{pq}\big]\big\}^{\frac{2}{q}}\d s\bigg\}^{\frac{1}{2}}.
\end{align}
Substituting \eqref{2.10} in \eqref{2.8}, and for any $q\geq 2,$  we find
\begin{align}\label{2.11}\nonumber
		\bigg\{\E\bigg[\bigg(|\Y^\e|_\H^\V(T)\bigg)^{pq}\bigg]\bigg\}^{\frac{2}{q}} & \leq 2 \|\x\|_\H^{2p}+C_p \bigg(\int_0^T\big(f(s)+g_1(s)\big)\d s\bigg)^2\\&\nonumber\quad+C_p\e^2 \bigg\{\E\bigg[\int_0^T
	\big(f(\e s)+g(\e s)\big)\big(|\Y^\e|_\H^\V(s)\big)^p\d s\bigg]^q\bigg\}^\frac{2}{q}\\&\quad +C_p q\e T\bigg(\esssup\limits_{s\in[0,T]}g_1(s)\bigg)+C_pq\e\int_0^T g_1(\e s)\big\{\E\big[ \|\Y^\e(s)\|_\H^{pq}\big]\big\}^{\frac{2}{q}}\d s.
\end{align}We consider the penultimate term in the right hand side of the  inequality \eqref{2.11} and estimate it using Minkowski's  and H\"older's inequalities as 
\begin{align}\label{2.12}\nonumber
&C_p\e^2 \bigg\{\E\bigg[\int_0^T
	\big(f(\e s)+g_1(\e s)\big)\big(|\Y^\e|_\H^\V(s)\big)^p\d s\bigg]^q\bigg\}^\frac{2}{q}\\&\nonumber\leq C_p\e^2\left[\int_0^T\big(f(\e s)+g_1(\e s)\big)\bigg\{\E\bigg[\bigg(|\Y^\e|_\H^\V(s)\bigg)^{pq}\bigg]\bigg\}^{\frac{1}{q}}\d s\right]^2\\&\leq C_p\e\bigg(\int_0^T\big(f(s)+g_1(s)\big)\d s\bigg)\bigg(\int_0^T\big(f(\e s)+g_1(\e s)\big)\bigg\{\E\bigg[\bigg(|\Y^\e|_\H^\V(s)\bigg)^{pq}\bigg]\bigg\}^{\frac{2}{q}}\d s\bigg).
\end{align}Substituting \eqref{2.12} in \eqref{2.11}, we deduce
\begin{align}\label{2.13}\nonumber
	&	\bigg\{\E\bigg[\bigg(|\Y^\e|_\H^\V(T)\bigg)^{pq}\bigg]\bigg\}^{\frac{2}{q}} \\&\nonumber \leq 2\|\x\|_\H^{2p}+C_p \bigg(\int_0^T\big(f(s)+g_1(s)\big)\d s\bigg)^2\\&\nonumber\quad +C_p\e \bigg(\int_0^T\big(f(s)+g_1(s)\big)\d s\bigg)\bigg(\int_0^T\big(f(\e s)+g_1(\e s)\big)\bigg\{\E\bigg[\bigg(|\Y^\e|_\H^\V(s)\bigg)^{pq}\bigg]\bigg\}^{\frac{2}{q}}\d s\bigg)\\&\nonumber\quad + C_p q\e T\bigg(\esssup\limits_{s\in[0,T]}g_1(s)\bigg)+ C_pq\e\int_0^T g_1(\e s)\big\{\E\big[ \|\Y^\e(s)\|_\H^{pq}\big]\big\}^{\frac{2}{q}}\d s	\\&\nonumber \leq 2\|\x\|_\H^{2p}+C_p \bigg(\int_0^T\big(f(s)+g_1(s)\big)\d s\bigg)^2\\&\nonumber\quad +C_p\e \bigg(\int_0^T\big(f(s)+g_1(s)\big)\d s\bigg)\bigg(\int_0^T\big(f(\e s)+g_1(\e s)\big)\bigg\{\E\bigg[\bigg(|\Y^\e|_\H^\V(s)\bigg)^{pq}\bigg]\bigg\}^{\frac{2}{q}}\d s\bigg)\\&\quad +C_pq\e T\bigg(\esssup\limits_{s\in[0,T]}g_1(s)\bigg)+C_pq\e\int_0^T g_1(\e s)\bigg\{\E\bigg[\bigg(|\Y^\e|_\H^\V(s)\bigg)^{pq}\bigg]\bigg\}^{\frac{2}{q}}\d s.
\end{align}An application of Gronwall's inequality  in \eqref{2.13} yields
\begin{align}\label{2.14}\nonumber
		&	\bigg\{\E\bigg[\bigg(|\Y^\e|_\H^\V(T)\bigg)^{pq}\bigg]\bigg\}^{\frac{2}{q}} \\&\nonumber \leq \bigg\{ 2\|\x\|_\H^{2p}+C_p \bigg(\int_0^T\big(f(s)+g_1(s)\big)\d s\bigg)^2+C_p q\e T\bigg(\esssup\limits_{s\in[0,T]}g_1(s)\bigg)  \bigg\}\\&\qquad\times \exp\bigg\{C_p  \bigg(\int_0^T\big(f(s)+g_1(s)\big)\d s\bigg)^2+C_pq\e T\bigg(\esssup\limits_{s\in[0,T]}g_1(s)\bigg) \bigg\}.
\end{align}An application of Markov's inequality yields 
\begin{align*}
	\P \bigg\{\bigg(|\Y^\e|_\H^\V(T)\bigg)^p>R\bigg\} \leq R^{-q}\E\bigg[\bigg(|\Y^\e|_\H^\V(T)\bigg)^{pq}\bigg].
\end{align*}Taking $q=\frac{2}{\e}$, we get
\begin{align}\label{2.15}\nonumber
	&\e\log\P \bigg\{\bigg(|\Y^\e|_\H^\V(T)\bigg)^p>R\bigg\}  \\&\nonumber\leq -2\log R+\log \bigg\{\E\bigg[\bigg(|\Y^\e|_\H^\V(T)\bigg)^{pq}\bigg]\bigg\}^\frac{2}{q}\\&\nonumber\leq -2\log R+\log\bigg\{ 2\|\x\|_\H^{2p}+C_p \bigg(\int_0^T\big(f(s)+g_1(s)\big)\d s\bigg)^2+2C_p T\bigg(\esssup\limits_{s\in[0,T]}g_1(s)\bigg) \bigg\}\\&\quad + C_p  \bigg(\int_0^T\big(f(s)+g_1(s)\big)\d s\bigg)^2+2C_p T  \bigg(\esssup\limits_{s\in[0,T]}g_1(s)\bigg).
\end{align}Passing $R\to\infty$ and using the fact that $f\in\L^1(0,T;\R^+)$ and $g_1\in\L^\infty(0,T;\R^+)$ in the above inequality \eqref{2.15}, we obtain the required result \eqref{2.5}. 
\end{proof}Since the embedding $\V\hookrightarrow \H$ is dense, there exists a sequence $\{\x_m\}_{m\in\N}$ in $\V$ such that 
\begin{align}\label{2.16}
	\lim_{m\to\infty} \|\x_m-\x\|_\H=0.
\end{align}Let $\Y_m^\e(\cdot)$ be the solution of \eqref{2.2} with the initial data $\x_m$.
 By Lemma \ref{lemma1}, we have
\begin{align}\label{2.17}
	\lim_{R\to\infty}\sup_{m\in\N}\sup_{\e\in(0,1]}\e\log\P\bigg\{\bigg(|\Y_m^\e|_\H^\V(T)\bigg)^p>R\bigg\}=-\infty.
\end{align}Let us denote $\X_m^\e(\cdot)$ as the solution of \eqref{2.3}  with the initial data $\x_m$, that is, 
	\begin{align}\label{2.03}
	\X_m^\e(t)=\x_m+\sqrt{\e}\int_0^t\B(\e s,\X_m^\e(s))\d \W(s).
\end{align}
The following lemma plays a crucial role in the proof of our main result. 
\begin{lemma}\label{lemma2}
	For any $m\in\mathbb{N}$, we have 
	\begin{align}\label{2.18}
		\lim_{R\to\infty}\sup_{\e\in(0,1]}\e\log\P \bigg\{\sup_{t\in[0,T]}\|\X_m^\e(t)\|_\V^2>R\bigg\}=-\infty.
	\end{align}
\end{lemma}
\begin{proof}To estimate the stochastic integral in $\V$-norm, we use the  Burkholder-Davis-Gundy type inequality (see Theorem 1.1, \cite{JS}) for 2-smooth Banach spaces. For any $q>1$, we have
	\begin{align}\label{2.19}		\nonumber
	&	\bigg\{\E\bigg[\sup_{t\in[0,T]}\|\X_m^\e(t)\|_\V^{2q}\bigg]\bigg\}^{\frac{1}{2q}} \\&\nonumber\leq \|\x_m\|_\V +\sqrt{\e}\bigg\{\E\bigg[\sup_{t\in[0,T]}\bigg\|\int_0^t\B(\e s,\X_m^\e(s))\d \W(s)\bigg\|_\V^{2q}\bigg]\bigg\}^{\frac{1}{2q}}\\&\nonumber\leq \|\x_m\|_\V +C\sqrt{2q\e}\bigg\{\E\bigg[\int_0^T\|\B(\e s,\X_m^\e(s))\|_\gamma^2\d s\bigg]^q\bigg\}^{\frac{1}{2q}}
	\\&\nonumber 
	\leq \|\x_m\|_\V +C\sqrt{2q\e}\bigg\{\E\bigg[\int_0^Th(\e s)\big(1+\|\X_m^\e(s)\|_\V^2\big)\d s\bigg]^q\bigg\}^{\frac{1}{2q}}\\&\nonumber
	\leq \|\x_m\|_\V +C\sqrt{2q\e T}\bigg(\esssup_{s\in[0,T]}h(s)\bigg)^{\frac{1}{2}}+C\sqrt{2q\e}\left[\bigg\{\E\bigg[\int_0^Th(\e s)\|\X_m^\e(s)\|_\V^2\d s\bigg]^q\bigg\}^{\frac{1}{q}}\right]^{\frac{1}{2}}
	\\&
	\leq \|\x_m\|_\V +C\sqrt{2q\e T}\bigg(\esssup_{s\in[0,T]}h(s)\bigg)^{\frac{1}{2}}+C\sqrt{2q\e}\left[\int_0^Th(\e s)\left\{\E\bigg[\|\X_m^\e(s)\|_\V^{2q}\bigg]\right\}^{\frac{1}{q}}\d s\right]^{\frac{1}{2}},
	\end{align}
where we have used Hypothesis \ref{hyp2} (H.6), Minkowski's and H\"older's inequalities, and the constant $C$ is independent of $\e$ and $q$. Therefore, we obtain 
	\begin{align}\label{2.20}	\nonumber	&\bigg\{\E\bigg[\sup_{t\in[0,T]}\|\X_m^\e(t)\|_\V^{2q}\bigg]\bigg\}^{\frac{1}{q}} 
		\\&	\leq2 \|\x_m\|_\V^2 +Cq\e T\bigg(\esssup_{s\in[0,T]}h(s)\bigg)+Cq\e \bigg(\int_0^Th(\e s )\bigg\{\E\bigg[\|\X_m^\e(s)\|_\V^{2q}\bigg]\bigg\}^{\frac{1}{q}}\d s\bigg).
\end{align}
An application of Gronwall's inequality in \eqref{2.20} yields
\begin{align}\label{2.21}\nonumber
	\bigg\{\E\bigg[\sup_{t\in[0,T]}\|\X_m^\e(t)\|_\V^{2q}\bigg]\bigg\}^{\frac{1}{q}}& \leq \bigg\{2 \|\x_m\|_\V^2+Cq\e T\bigg(\esssup_{s\in[0,T]}h(s)\bigg)\bigg\}\\&\qquad\times\exp\bigg\{Cq\e T\bigg(\esssup_{s\in[0,T]}h(s)\bigg)\bigg\}.
\end{align}Fixing $R$ and taking $q=\frac{1}{\e}$ and applying Markov's inequality, we get
\begin{align}\label{2.22}	\nonumber
&\e\log\P\bigg\{\sup_{t\in[0,T]}\|\X_m^\e(t)\|_\V^2>R\bigg\}	 \\&	\nonumber\leq \e \log \bigg\{R^{-q}\E\bigg[\sup_{t\in[0,T]}\|\X_m^\e(t)\|_\V^{2q}\bigg]\bigg\}\\&	\nonumber\leq  -\log R +\log \bigg\{\E\bigg[\sup_{t\in[0,T]}\|\X_m^\e(t)\|_\V^{2q}\bigg]\bigg\}^{\frac{1}{q}} \\&\leq -\log R+\log\bigg\{2 \|\x_m\|_\V^2+C T\bigg(\esssup_{s\in[0,T]}h(s)\bigg)\bigg\}+C T\bigg(\esssup_{s\in[0,T]}h(s)\bigg). 
\end{align}Passing $R\to\infty$ and  using the fact that $h\in\L^\infty(0,T;\R^+)$ and \eqref{2.16}, we obtain the required result \eqref{2.18}.
\end{proof}
Before going to the exponential convergence part, let us establish the following lemma:
\begin{lemma}\label{lemma3}
	Under Hypotheses \ref{hyp1} and \ref{hyp2}, for the solution $\Y_m^\e(\cdot)$ of the system \eqref{1.1} corresponding to the initial data $\x_m\in\H$, we have
	\begin{align}\label{lem3.1}
\sup_{m\in\N} \E\bigg[\sup_{t\in[0,T]}\|\Y_m^\e(t)\|_\H^p+\int_0^T\|\Y_m^\e(t)\|_\H^{p-2}\|\Y_m^\e(t)\|_\V^\beta\d t\bigg]		 \leq C,
	\end{align}where the constant $C$ is independent of $m$.
\end{lemma}
\begin{proof}
	Applying infinite dimensional It\^o's formula to the process $\|\Y_m^\e(\cdot)\|_\H^p$, we find
		\begin{align}\label{lem3.2}\nonumber
		&	\|\Y_m^\e(t)\|_\H^p +\frac{pL_\A\e}{2}\int_0^t \|\Y_m^\e(s)\|_\H^{p-2}\|\Y_m^\e(s)\|_\V^\beta\d s\\&\nonumber\leq \|\x_m\|_\H^p +C_p  \int_0^{T}\big(f(s)+g_1(s)\big)\d s+C_p\e \int_0^t\big(f(\e s)+g_1(\e s)\big)\|\Y_m^\e(s)\|_\H^p\d s\\&\quad + p\sqrt{\e}\int_0^t \|\Y_m^\e(s)\|_\H^{p-2}\big(\B(\e s,\Y_m^\e(s))\d \W(s),\Y_m^\e(s)\big),
	\end{align}where we  have used similar arguments as in  \eqref{2.6} and \eqref{2.7}. 

Define the following sequence of stopping times  for $R>0$
\begin{align*}
	\tau_{m,R}^\e:=\inf\{t\geq 0: \|\Y_m^\e(t)\|_\H>R\}\wedge T.
\end{align*}
Then, $\tau_{m,R}^\e \to T,\ \P$-a.s., as $R\to\infty$, for every $m\in\mathbb{N}$. Taking supremum  from 0 to $t\leq r\wedge \tau_{m,R}^\e$ and then taking expectation in \eqref{lem3.2}, we obtain
\begin{align}\label{lem3.3}\nonumber
&	\E\bigg[\sup_{t\in[0,r\wedge \tau_{m,R}^\e]}\|\Y_m^\e(t)\|_\H^p\bigg]+\frac{p L_\A\e }{2}\E\bigg[\int_0^{r\wedge \tau_{m,R}^\e}\|\Y_m(s)\|_\H^{p-2}\|\Y_m(s)\|_\V^\beta\d s\bigg]\\& \nonumber\leq \|\x_m\|_\H^p +C_p \int_0^{ T}\big(f(s)+g_1(s)\big)\d s+C_p\e\E\bigg[ \int_0^{r\wedge \tau_{m,R}^\e}\big(f(\e s)+g_1(\e s)\big)\|\Y_m^\e(s)\|_\H^p\d s\bigg]\\&\quad +\underbrace{ p\sqrt{\e}\E\bigg[\sup_{t\in[0,r\wedge \tau_{m,R}^\e]}\bigg|\int_0^t \|\Y_m^\e(s)\|_\H^{p-2}\big(\B(\e s,\Y_m^\e(s))\d \W(s),\Y_m^\e(s)\big)\bigg|\bigg]}_{=:I(t)}.
\end{align}We consider the final term $I(\cdot)$ of the right hand side of the above inequality \eqref{lem3.3} and estimate it using the Burkholder-Davis-Gundy inequality (see Theorem 1.1, \cite{DLB}), Hypothesis \ref{hyp2} (H.6), H\"older's and Young's inequalities as
\begin{align}\label{lem3.4}\nonumber
	I(t)& \leq C_p\sqrt{\e} \E\bigg[\int_0^{r\wedge \tau_{m,R}^\e}\|\Y_m^\e(s)\|_\H^{2p-2}\|\B(\e s,\Y_m^\e(s))\|_{\L_2}^2\d s\bigg]^{\frac{1}{2}} \\&\nonumber\leq C_p\sqrt{\e} \E\bigg[\sup_{s\in[0,r\wedge \tau_{m,R}^\e]}\|\Y_m^\e(s)\|_\H^p\int_0^{r\wedge \tau_{m,R}^\e}\|\Y_m^\e(s)\|_\H^{p-2}\|\B(\e s,\Y_m^\e(s))\|_{\L_2}^2\d s\bigg]^\frac{1}{2} \\&\nonumber\leq  \frac{1}{2} \E\bigg[\sup_{s\in[0,r\wedge \tau_{m,R}^\e]}\|\Y_m^\e(s)\|_\H^p\bigg]+C_p\e \E\bigg[\int_0^{r\wedge \tau_{m,R}^\e}\|\Y_m^\e(s)\|_\H^{p-2}\|\B(\e s,\Y_m^\e(s))\|_{\L_2}^2\d s\bigg]\\&\leq  \frac{1}{2} \E\bigg[\sup_{s\in[0,r\wedge \tau_{m,R}^\e]}\|\Y_m^\e(s)\|_\H^p\bigg]+C_p\int_0^Tg_1(s)\d s+C_p\e\bigg( \E\bigg[\int_0^{r\wedge \tau_{m,R}^\e}g_1(\e s)\|\Y_m^\e(s)\|_\H^p\d s\bigg]\bigg).
\end{align}Substituting \eqref{lem3.4} in \eqref{lem3.3}, we deduce 
\begin{align}\label{lem3.5}\nonumber
	&	\E\bigg[\sup_{t\in[0,r\wedge \tau_{m,R}^\e]}\|\Y_m^\e(t)\|_\H^p\bigg]+\frac{p L_\A\e }{2}\E\bigg[\int_0^{r\wedge \tau_{m,R}^\e}\|\Y_m(s)\|_\H^{p-2}\|\Y_m(s)\|_\V^\beta\d s\bigg]\\& \leq 
2	\|\x_m\|_\H^p +C_p \int_0^{ t}\big(f(s)+g_1(s)\big)\d s+C_p\e\E\bigg[ \int_0^{r\wedge \tau_{m,R}^\e}\big(f(\e s)+g_1(\e s)\big)\|\Y_m^\e(s)\|_\H^p\d s\bigg]. %\\&\quad +C_p\e \bigg(\int_0^Tg(s)\d s+ \E\bigg[\int_0^{r\wedge \tau_{m,R}^\e}g(\e s)\|\Y_m^\e(s)\|_\H^p\d s\bigg]\bigg).
\end{align}Applying Gronwall's inequality in  \eqref{lem3.5} and substituting back in \eqref{lem3.5}, we deduce
\begin{align}\label{lem3.6}\nonumber
	&	\E\bigg[\sup_{t\in[0,r\wedge \tau_{m,R}^\e]}\|\Y_m^\e(t)\|_\H^p\bigg]+\frac{p L_\A\e }{2}\E\bigg[\int_0^{r\wedge \tau_{m,R}^\e}\|\Y_m(s)\|_\H^{p-2}\|\Y_m(s)\|_\V^\beta\d s\bigg]
	 \\&\leq \bigg(2	\|\x_m\|_\H^p +C_p  \int_0^{ T}\big(f(s)+g_1(s)\big)\d s \bigg)e^{C_p\int_0^{ T}(f(s)+g_1(s))\d s }\leq C,
	\end{align}where we have used the fact that  $f\in \L^1(0,T;\R^+)$ and $g_1\in\L^\infty(0,T;\R^+)$. 

Now our aim is show that the constant $C$ is independent of $m$. In view of \eqref{lem3.6}, it is enough to show that the bound of $\|\x_m\|_\H$ is independent of $m$.   We know that the sequence $\{\x_m\}_{m\in\N}$ is in the space $\V$ with $\lim\limits_{m\to\infty}\|\x_m-\x\|_\H=0$, which implies for any given $\delta>0$, there exists a natural number $M$ such that $$\|\x_m-\x\|_\H\leq \delta, \ \text{ for all } \ m\geq M.$$ 
Then, for all $m\geq M$, we have
\begin{align}\label{lem3.7}
	\|\x_m\|_\H \leq \|\x_m-\x\|_\H+\|\x\|_\H \leq \delta+\|\x\|_\H <\infty.
\end{align}For $m< M$, we have 
\begin{align}\label{lem3.8}
	\|\x_m\|_\H \leq \|\x_m\|_\H+2\|\x\|_\H \leq \max_{1\leq m\leq M}\|\x_m\|_\H+2\|\x\|_\H<\infty.
\end{align}Combining \eqref{lem3.6}-\eqref{lem3.8}, we conclude that the constant $C$ is independent of $m$. Passing $M\to\infty$, and applying Fatou's lemma, we obtain the required result \eqref{lem3.1}.
\end{proof}

Let us establish the exponential convergence of $\Y_m^\e(\cdot)-\Y^\e(\cdot)$.
\begin{lemma}\label{lemma4}
For any $\varpi>0$, we have 
\begin{align}\label{2.23}
	\lim_{m\to\infty}\sup_{\e\in(0,1]}\e\log\P\bigg\{\sup_{t\in[0,T]}\|\Y_m^\e(t)-\Y^\e(t)\|_\H^2>\varpi\bigg\}=-\infty.
\end{align}
\end{lemma}
\begin{proof}
	For any $R>0$, we define a sequence of stopping times as 
	\begin{align}\label{2.24}	\nonumber
		\tau_{1,R}^\e&:=\inf\bigg\{t\geq 0:\e\int_0^t\|\Y^\e(s)\|_\H^{p-2}\|\Y^\e(s)\|_\V^\beta\d s>R,\  \text{ or }\  \|\Y^\e(t)\|_\H^2>R\bigg\} ,\\
			\tau_{2,R}^\e&:=\inf\bigg\{t\geq 0: \e\int_0^t\|\Y_m^\e(s)\|_\H^{p-2}\|\Y_m^\e(s)\|_\V^\beta\d s>R, \ \text{ or }\  \|\Y_m^\e(t)\|_\H^2>R\bigg\},
	\end{align}where the sequence of stopping times $\tau_{2,R}^\e$ is possible because of Lemma \ref{lemma3}. 

Set $\tau_R^\e =\tau_{1,R}^\e\wedge \tau_{2,R}^\e$. Applying infinite dimensional It\^o's formula to the process $\phi_m(\cdot)\|\Y^\e(\cdot)-\X^\e(\cdot)\|_\H^2$ (see Theorem 1.2, \cite{IGDS}), where $\phi_m(\cdot)$ is given by $$\phi_m(t)=\exp\bigg(-\e\int_0^t\big[f(\e s)+\rho(\Y_m^\e(s))+\eta(\Y^\e(s))\big]\d s\bigg),$$ to find 
\begin{align}\label{2.25}	\nonumber
	&\phi_m(\t)\|\Y_m^\e(t)-\Y^\e(t)\|_\H^2\\&	\nonumber= \|\x_m-\x\|_\H^2+\e\int_0^{\t}\phi_m(s)\bigg(2\langle \A(\e s,\Y_m^\e(s))-\A(\e s,\Y^\e(t)),\Y_m^\e(s)-\Y^\e(s)\rangle\\&\qquad	\nonumber +\| \B(\e s,\Y_m^\e(s))-\B(\e s,\Y^\e(s))\|_{\L_2}^2-\big[f(\e s)+\rho(\Y_m^\e(s))+\eta(\Y^\e(s))\big]\bigg)\d s\\&\nonumber\quad +2\sqrt{\e}\int_0^{\t}\phi_m(s)\big((\B(\e s,\Y_m^\e(s))-\B(\e s,\Y^\e(s)))\d \W(s),\Y_m^\e(s)-\Y^\e(s)\big) \\&\nonumber\leq 
	\|\x_m-\x\|_\H^2 \\&\quad+2\sqrt{\e}\int_0^{\t}\phi_m(s)\big((\B(\e s,\Y_m^\e(s))-\B(\e s,\Y^\e(s)))\d \W(s),\Y_m^\e(s)-\Y^\e(s)\big),
\end{align}where we have used Hypothesis \ref{hyp1} (H.2). Moreover, we have 
\begin{align}\label{2.26}\nonumber
&	\bigg\{\E\bigg[\sup_{s\in[0,\t]} \phi_m(s)\|\Y_m^\e(s)-\Y^\e(s)\|_\H^2\bigg]^q\bigg\}^\frac{2}{q} \\&\nonumber\leq 2\|\x_m-\x\|_\H^4\\&\quad+\underbrace{4\e \bigg\{\E\bigg[\sup_{s\in[0,\t]}\bigg|\int_0^s \phi_m(s)\big((\B(\e s,\Y_m^\e(s))-\B(\e s,\Y^\e(s)))\d \W(s),\Y_m^\e(s)-\Y^\e(s)\big)\bigg| \bigg]^q\bigg\}^{\frac{2}{q}}}_{=:J(t)}
\end{align}
Using the result \eqref{2.9}, Hypothesis \ref{hyp2} (H.6), Minkowski's, H\"older's and Young's inequalities in the final term $J(\cdot)$ of the above inequality \eqref{2.26},  we arrive at 
\begin{align}\label{2.27}\nonumber
	J(t) &\leq 4C\e\sqrt{q} \bigg\{\E\bigg[\int_0^{\t} \phi_m^2(s)\|\B(\e s,\Y_m^\e(s))-\B(\e s,\Y^\e(s))\|_{\L_2}^2\|\Y_m^\e(s)-\Y^\e(s)\|_\H^2\d s\bigg]^{\frac{q}{2}}\bigg\}^{\frac{2}{q}}\\&\nonumber\leq 4C\e\sqrt{q} \bigg\{\E\bigg[\sup_{s\in[0,\t]}\phi_m(s)\|\Y_m^\e(s)-\Y^\e(s)\|_\H^2\\&\nonumber\qquad\times\int_0^{\t} \phi_m(s)\|\B(\e s,\Y_m^\e(s))-\B(\e s,\Y^\e(s))\|_{\L_2}^2\d s\bigg]^{\frac{q}{2}}\bigg\}^{\frac{2}{q}}\\&\nonumber\leq  \frac{1}{2}\bigg\{\E\bigg[\sup_{s\in[0,\t]}\phi_m(s)\|\Y_m^\e(s)-\Y^\e(s)\|_\H^2\bigg]^q\bigg\}^\frac{2}{q}\\&\nonumber\quad+8Cq\e^2 \bigg\{\E\bigg[\int_0^{\t}\phi_m(s)\|\B(\e s,\Y_m^\e(s))-\B(\e s,\Y^\e(s))\|_{\L_2}^2\d s\bigg]^q\bigg\}^{\frac{2}{q}}
	\\&\nonumber\leq  \frac{1}{2}\bigg\{\E\bigg[\sup_{s\in[0,\t]}\phi_m(s)\|\Y_m^\e(s)-\Y^\e(s)\|_\H^2\bigg]^q\bigg\}^\frac{2}{q}\\&\nonumber\quad+8Cq\e^2 \bigg(\int_0^{\t}L_\B(\e s )\bigg\{\E\bigg[\phi_m(s)\|\Y_m^\e(s)-\Y^\e(s)\|_\H^2\bigg]^q\bigg\}^{\frac{1}{q}}\d s\bigg)^2 \\&\nonumber\leq \frac{1}{2}\bigg\{\E\bigg[\sup_{s\in[0,\t]}\phi_m(s)\|\Y_m^\e(s)-\Y^\e(s)\|_\H^2\bigg]^q\bigg\}^\frac{2}{q}\\&\quad+8Cq\e \bigg(\int_0^{T}L_\B(s)\d s\bigg) \bigg(\int_0^{t}L_\B(\e s ) \bigg\{\E\bigg[\sup_{r\in[0,\s]}\phi_m(s)\|\Y_m^\e(s)-\Y^\e(s)\|_\H^2\d s\bigg]^q\bigg\}^{\frac{2}{q}}\d s\bigg).
\end{align}
Substituting \eqref{2.27} in \eqref{2.26}, and then applying Gronwall's inequality, we deduce
\begin{align}\label{2.28}\nonumber
	&	\bigg\{\E\bigg[\sup_{s\in[0,\t]} \phi_m(s)\|\Y_m^\e(s)-\Y^\e(s)\|_\H^2\bigg]^q\bigg\}^\frac{2}{q} \\&\leq 4\|\x_m-\x\|_\H^4\exp\bigg\{16Cq\e T^2 \bigg(\esssup_{s\in[0,T]}L_\B(s)\bigg)^2\bigg\}.
	\end{align}Therefore, we have 
\begin{align}\label{2.29}\nonumber
		&	\bigg\{\E\bigg[\sup_{s\in[0,\t]} \|\Y_m^\e(s)-\Y^\e(s)\|_\H^{2q}\bigg]\bigg\}^\frac{2}{q} \\&\nonumber\leq 
		\left\{\E\bigg[\sup_{s\in[0,\T]}\bigg\{e^{-\e\int_0^t[f(\e s)+\rho(\Y_m^\e(s))+\eta(\Y^\e(s))]\d s} \|\Y_m^\e(s)-\Y^\e(s)\|_\H^2  \bigg\}^q \right.  \\&\nonumber \left.\qquad \times e^{q\e\int_0^t[f(\e s)+\rho(\Y_m^\e(s))+\eta(\Y^\e(s))]\d s}\bigg]\right\}^\frac{2}{q}
		\\&\nonumber\leq 
		e^{2\big[\e\int_0^Tf(\e s)\d s+C(1+R^\frac{\zeta}{2})(\e T+R)\big]}\bigg\{\E\bigg[\sup_{s\in[0,\T]}e^{-\e\int_0^t[f(\e s)+\rho(\Y_m^\e(s))+\eta(\Y^\e(s))]\d s} \|\Y_m^\e(s)-\Y^\e(s)\|_\H^2  \bigg]^q \bigg\}^\frac{2}{q}		
		 \\&\leq 4\|\x_m-\x\|_\H^4\exp\bigg\{2 C(1+R^\frac{\zeta}{2})(\e T+R)+2\int_0^Tf(s)\d s+ 16Cq\e T^2 \bigg(\esssup_{s\in[0,T]}L_\B(s)\bigg)^2\bigg\},
\end{align}where we have used \eqref{1.3}, \eqref{2.28} and the definition of stopping times $\tau_R^\e$.

 Fixing $R$ and taking $q=\frac{2}{\e},$ and applying  Markov's  inequality, we conclude
\begin{align}\label{2.30}\nonumber
&	\sup_{\e\in(0,1]} \e\log\P\bigg\{\sup_{t\in[0,\T]}\|\Y_m^\e(t)-\Y^\e(t)\|_\H^2>\varpi\bigg\} \\&\nonumber\leq \sup_ {\e\in(0,1]} \e\log \bigg\{\varpi^{-q}\E\bigg[\sup_{t\in[0,\T]}\|\Y_m^\e(t)-\Y^\e(t)\|_\H^{2q}\bigg]\bigg\} \\&\nonumber\leq -2\log\varpi+\log\left\{ 4\|\x_m-\x\|_\H^4\right\}+ 2C(1+R^\frac{\zeta}{2})(T+R)+ \int_0^Tf(s)\d s \nonumber\\&\quad\nonumber+ 32C T^2 \bigg(\esssup_{s\in[0,T]}L_\B(s)\bigg)^2
\\&\to -\infty \ \text{ as } \ m\to\infty,
\end{align}where we have used the fact that $f\in \L^1(0,T;\R^+)$ and  $L_\B\in \L^\infty(0,T;\R^+)$ and \eqref{2.16}. 

Applying Lemma \ref{lemma1}, we find for any $M>0$, there exists a constant $R>0$ such that for any $\e\in(0,1]$, the following holds:
\begin{align}\label{2.31}
\P\bigg\{\bigg(|\Y^\e|_\H^\V(T)\bigg)^p>R\bigg\}\leq e^{-\frac{M}{\e}}.	
\end{align}For such a constant $R$, combining \eqref{2.30} and the definition of stopping times \eqref{2.24}, there exists an integrer $N>0$, such that for any $m\geq N$, 
\begin{align}\label{2.32}\nonumber
	&\sup_{\e\in(0,1]}\e \log\P \bigg\{\sup_{t\in[0,T]}\|\Y_m^\e(t)-\Y^\e(t)\|_\H^2>\varpi,\; \bigg(|\Y^\e|_\H^\V(T)\bigg)^p\leq R\bigg\} \\&\leq \sup_{\e\in(0,1]}\e \log\P \bigg\{\sup_{t\in[0,\T]}\|\Y_m^\e(t)-\Y^\e(t)\|_\H^2>\varpi\bigg\} \leq -M.
\end{align}
Combining \eqref{2.31} and \eqref{2.32}, we conclude that there exists an integer $N>0$ such that for any $m\geq N$, $\e\in (0,1]$, 
\begin{align}\label{2.33}
	&\sup_{\e\in(0,1]}\e \log\P \bigg\{\sup_{t\in[0,T]}\|\Y_m^\e(t)-\Y^\e(t)\|_\H^2>\varpi\bigg\} \leq -2 M.
\end{align}
Since the choice of the constant $M$ is arbitrary, the proof follows.
\end{proof}
\begin{lemma}\label{lemma5}
	For any $\varpi>0$, we have 
	\begin{align}\label{2.34}
		\lim_{m\to\infty} \sup_{\e\in(0,1]} \e\log \P\bigg\{\|\X_m^\e(t)-\X^\e(t)\|_\H^2>\varpi\bigg\}=-\infty.
	\end{align}
\end{lemma}
\begin{proof}
	From \eqref{2.3} and \eqref{2.03}, we have 
	\begin{align}\label{2.35}
		\X_m^\e(t)-\X^\e(t)=\x_m-\x+\sqrt{\e} \int_0^t\big(\B(\e s,\X_m^\e(s))-\B(\e s,\X^\e(s))\big)\d \W(s).
	\end{align}
Applying infinite dimensional It\^o's formula to the proces $\|	\X_m^\e(\cdot)-\X^\e(\cdot)\|_\H^2$, we find 
\begin{align}\label{2.36}\nonumber
	\|	\X_m^\e(t)-\X^\e(t)\|_\H^2&= \|\x_m-\x\|_\H^2+\e\int_0^t \|\B(\e s,\X_m^\e(s))-\B(\e s,\X^\e(s))\|_{\L_2}^2\d s\\&\quad + 2\sqrt{\e}\int_0^t\big(\B(\e s,\X_m^\e(s))-\B(\e s,\X^\e(s))\d \W(s),\X_m^\e(s)-\X^\e(s) \big).
\end{align}Then, for any $q>1$, we have 
\begin{align}\label{2.37}\nonumber
&	\bigg\{\E\bigg[	\|	\X_m^\e(t)-\X^\e(t)\|_\H^{2q}\bigg]\bigg\}^\frac{2}{q} \\&\leq 4\|\x_m-\x\|_\H^4+16Cq\e \bigg(\int_0^TL_\B(s)\d s\bigg)\bigg(\int_0^TL_\B(\e s)\bigg\{\E\bigg[\|	\X_m^\e(s)-\X^\e(s)\|_\H^{2q}\bigg]\bigg\}^\frac{2}{q}\d s\bigg),
\end{align}
where we have used a  similar calculation as in \eqref{2.27} and the constant $C$ is independent of $\e$ and $q$.

An application of Gronwall's inequality in \eqref{2.37} yields
\begin{align}\label{2.38}
	\bigg\{\E\bigg[	\|	\X_m^\e(t)-\X^\e(t)\|_\H^{2q}\bigg]\bigg\}^\frac{2}{q} \leq 4\|\x_m-\x\|_\H^4 \exp\bigg\{16Cq\e T^2 \bigg(\esssup_{s\in[0,T]}L_\B(s)\bigg)^2\bigg\}.
\end{align}Using the same arguments as in the proof of \eqref{2.30} in Lemma \ref{lemma4}, we obtain the required result \eqref{2.34}. 
\end{proof}
The following lemma deals with the exponential equivalency of two families $\{\Y_m^\e\}$ and $\{\X_m^\e\}$.
\begin{lemma}\label{lemma6}
	For any $\varpi>0$, and any $m\in\N$, we have 
	\begin{align}\label{2.39}
		\lim_{\e\to0}\e\log\P\bigg\{\sup_{t\in[0,T]}\|\Y_m^\e(t)-\X_m^\e(t)\|_\H^2>\varpi\bigg\}=-\infty.
	\end{align}
\end{lemma}
\begin{proof}
	For any $R>0$, we define the following sequence of stopping times:
	\begin{align*}
		\tau_{R,\e}^{1,m}&:=\inf\bigg\{t\geq0:\e\int_0^t\|\Y_m^\e(s)\|_\H^{p-2}\|\Y_m^\e(s)\|_\V^\beta\d s>R, \ \text{ or }\  \|\Y_m^\e(t)\|_\H^2>R\bigg\},\\
			\tau_{R,\e}^{2,m}&:=\inf\big\{t\geq0: \|\X_m^\e(t)\|_\V^2>R\big\}.
	\end{align*}
Setting $\tau_{R,\e}^m=\tau_{R,\e}^{1,m}\wedge	\tau_{R,\e}^{2,m}$, and applying It\^o's formula to the process $\|\Y_m^\e(\cdot) -\X_m^\e(\cdot)\|_\H^2$, we find
\begin{align}\label{2.40}\nonumber
&\|\Y_m^\e(\tt) -\X_m^\e(\tt)\|_\H^2 \\&\nonumber= \e\int_0^{\tt}2\langle \A(\e s,\X_m^\e(s)),\Y_m^\e(s) -\X_m^\e(s)\rangle\d s\\&\nonumber \quad +\e\int_0^{\tt}\big[2\langle \A(\e s, \Y_m^\e(s))-\A(\e s,\X_m^\e(s)),\Y_m^\e(s) -\X_m^\e(s)\rangle \d s \\&\nonumber\qquad+\|\B(\e s,\Y_m^\e(s))-\B(\e s,\X_m^\e(s))\|_{\L_2}^2\big] \d s\\&\quad +2\sqrt{\e} \int_0^{\tt}\big((\B(\e s,\Y_m^\e(s))-\B(\e s,\X_m^\e(s)))\d\W(s),\Y_m^\e(s) -\X_m^\e(s)\big). 
\end{align}Let us consider the  term $2\langle \A(\e \cdot,\X_m^\e),\Y_m^\e-\X_m^\e\rangle$, and estimate it using the Cauchy-Shwarz inequality, Hypothesis \ref{hyp1} (H.4), Young's inequality and the fact that $\V\hookrightarrow\H$ as 
\begin{align}\label{2.41}\nonumber
&\big|2\langle \A(\e\cdot,\X_m^\e),\Y_m^\e -\X_m^\e\rangle\big| \\&\nonumber\leq 2\|\A(\e\cdot,\X_m^\e)\|_{\V^*}\|\Y_m^\e -\X_m^\e\|_\V \\&\nonumber\leq 
C\|\A(\e\cdot,\X_m^\e)\|_{\V^*}^{\frac{\beta}{\beta-1}}+\theta \|\Y_m^\e -\X_m^\e\|_\V^\beta  \\&\nonumber\leq C\big(f(\e\cdot )+C\|\X_m^\e\|_\V^\beta\big)\big(1+\|\X_m^\e\|_\H^\alpha\big)+\theta \|\Y_m^\e -\X_m^\e\|_\V^\beta  \\&\leq C f(\e\cdot) \big(1+\|\X_m^\e\|_\V^\alpha\big)+C \|\X_m^\e\|_\V^\beta+C \|\X_m^\e\|_\V^{\beta+\alpha}+\theta \|\Y_m^\e -\X_m^\e\|_\V^\beta,
\end{align}for $\theta\in(0,\xi)$. Substituting \eqref{2.41} in \eqref{2.40}, and then using Hypothesis \ref{hyp2} (H.2)$^*$, we find 
\begin{align}\label{2.42}\nonumber
	&\|\Y_m^\e(\tt) -\X_m^\e(\tt)\|_\H^2 \\&\nonumber
	\leq \e\int_0^{\tt}\big[C f(\e s) \big(1+\|\X_m^\e(s)\|_\H^\alpha\big)+C \|\X_m^\e(s)\|_\V^\beta+C \|\X_m^\e(s)\|_\V^{\beta+\alpha} \big]\d s\\&\nonumber\quad +\e\int_0^{\tt}\bigg(\big(f(\e s)+\rho(\Y_m^\e(s))+\eta(\X_m^\e(s))\big)\|\Y_m^\e(s) -\X_m^\e(s)\|_\H^2\\&\nonumber\qquad- (\xi-\theta)\|\Y_m^\e(s) -\X_m^\e(s)\|_\V^\beta\bigg)\d s\\&\quad+2\sqrt{\e} \bigg| \int_0^{\tt}\big((\B(\e s,\Y_m^\e(s))-\B(\e s,\X_m^\e(s)))\d\W(s),\Y_m^\e(s) -\X_m^\e(s)\big) \bigg|. 
	\end{align} Applying Gronwall's inequality in the  inequality \eqref{2.42}, we deduce
\begin{align}\label{2.43}\nonumber
		&\|\Y_m^\e(\tt) -\X_m^\e(\tt)\|_\H^2 \\&\nonumber \leq \bigg\{C\e  \bigg[\int_0^{\tt}\big[ f(\e s) \big(1+\|\X_m^\e(s)\|_\V^\alpha\big)+ \|\X_m^\e(s)\|_\V^\beta+ \|\X_m^\e(s)\|_\V^{\beta+\alpha} \big]\d s\bigg]\\&\nonumber\qquad +2\sqrt{\e} \bigg| \int_0^{\tt}\big((\B(\e s,\Y_m^\e(s))-\B(\e s,\X_m^\e(s)))\d\W(s),\Y_m^\e(s) -\X_m^\e(s)\big) \bigg|\bigg\} \\&\quad \times \exp\bigg\{\e\int_0^{\tt}\big(f(\e s)+\rho(\Y_m^\e(s))+\eta(\X_m^\e(s))\big)\d s\bigg\}.
\end{align}
% Using the martingale inequality \eqref{2.9} and the definition of stopping times, we conclude that 
Then, for any $q>1$, we conclude that  
\begin{align}\label{2.44}\nonumber
	&	\bigg\{\E\bigg[	\sup_{s\in[0,\t]}\|	\Y_m^\e(t)-\X_m^\e(t)\|_\H^{2q}\bigg]\bigg\}^\frac{2}{q} \\&\nonumber\leq \left\{ C \e^2\bigg[\big(R^\frac{\alpha}{2}+1\big)\int_0^{T}f(\e s)\d s+R^\frac{\beta}{2}T+R^\frac{\alpha+\beta}{2}T\bigg]^2 \right.\\&\left.\nonumber\qquad+16Cq\e \bigg(\int_0^TL_\B(s)\d s\bigg)\bigg(\int_0^TL_\B(\e s)\bigg\{\E\bigg[\sup_{r\in[0,s\wedge \tau_{R,\e}^m]}\|	\Y_m^\e(r)-\X_m^\e(r)\|_\H^{2q}\bigg]\bigg\}^\frac{2}{q}\d s\bigg)\right\}\\&\qquad\times e^{\big[\int_0^{T}f(s)\d s+C(1+R^\frac{\zeta}{2})(\e T+R)+\e C_{R,T}^1\big]},
\end{align}where we have used a calculation similar to \eqref{2.27} and the fact that $\V\hookrightarrow\H$, and the constant $C_{R,T}^1=CT(1+R^\frac{\beta}{2}+R^\frac{\zeta}{2}+R^\frac{\zeta+\beta}{2})$.
Using Gronwall's inequality in \eqref{2.44}, we arrive at 
\begin{align}\label{2.45}\nonumber
	&	\bigg\{\E\bigg[	\sup_{s\in[0,\t]}\|	\Y_m^\e(t)-\X_m^\e(t)\|_\H^{2q}\bigg]\bigg\}^\frac{2}{q} \\&\nonumber\leq e^{\big[\int_0^{T}f(s)\d s+C(1+R^\frac{\zeta}{2})(\e T+R)+\e C_{R,T}^1\big]}  C \bigg[\big(R^\frac{\alpha}{2}+1\big)\int_0^{\e T}f(s)\d s+\e R^\frac{\beta}{2}T+\e R^\frac{\alpha+\beta}{2}T\bigg]^2\\&\qquad\times\exp\bigg\{16Cq\e T^2 \bigg(\esssup_{s\in[0,T]}L_\B(s)\bigg)^2\bigg\}.
\end{align}Fixing $R>0$, and taking $q=\frac{2}{\e},$ and applying  Markov's  inequality, we deduce
\begin{align}\label{2.46}\nonumber
&	\e\log \P\bigg\{\sup_{t\in[0,\TT]}\|\Y_m^\e(t)-\X_m^\e(t)\|_\H^2>\varpi \bigg\}\\&\nonumber\leq  \e \log\bigg\{\varpi^{-q}\E\bigg[\sup_{t\in[0,\TT]}\|\Y_m^\e(t)-\X_m^\e(t)\|_\H^{2q}\bigg]\bigg\} \\&\nonumber\leq -2\log\varpi+ \log \bigg\{C \bigg[\big(R^\frac{\alpha}{2}+1\big)\int_0^{\e T}f(s)\d s+ \e R^\frac{\beta}{2}T+\e R^\frac{\alpha+\beta}{2}T\bigg]^2\bigg\}+\int_0^{T}f(s)\d s\\&\nonumber\quad +C(1+R^\frac{\zeta}{2})( \e T+R)+ \e C_{R,T}^1+32C T^2 \bigg(\esssup_{s\in[0,T]}L_\B(s)\bigg)^2\\& \to -\infty \ \text{ as } \ \e \to0,
	\end{align} 
where we have used  the fact that  $f\in\L^1(0,T;\R^+)$, $L_\B\in\L^\infty(0,T;\R^+)$ and the absolute continuity of the Lebesgue integral.

Using \eqref{2.17} and Lemma \ref{lemma2}, for any $M>0$, there exists a  constant $R>0$ such that the following inequalities hold:
\begin{align}\label{2.47}
	\sup_{\e\in(0,1]} \e\log \P \bigg\{	\bigg(|\Y_m^\e|_\H^\V(T)\bigg)^p>R\bigg\} &\leq -M, \\ \label{2.48}
		\sup_{\e\in(0,1]}\e\log \P \bigg\{\sup_{t\in[0,T]}\|\X_m^\e(t)\|_\V^2>R\bigg\} &\leq -M.
\end{align}For such a constant $R>0$, using \eqref{2.46} and the definition of stopping time $\tau_{R,\e}^m$, there exists $\e_0>0$ such that for every $\e$ satisfying $\e\in(0,\e_0]$, 
\begin{align}\label{2.49}\nonumber
	&\P \bigg\{\sup_{t\in[0,T]}\|\Y_m^\e(t)-\X_m^\e(t)\|_\H^2>\varpi, \bigg(|\Y_m^\e|_\H^\V(T)\bigg)^p\leq R,\sup_{t\in[0,T]}\|\X_m^\e(t)\|_\V^2\leq R \bigg\}\\&\leq \P \bigg\{\sup_{t\in[0,\T]}\|\Y_m^\e(t)-\X_m^\e(t)\|_\H^2>\varpi\bigg\}\leq e^{-\frac{M}{\e}}.
\end{align}Now, combining \eqref{2.47}-\eqref{2.49}, we conclude that there exists $\e_0$, such that for every $\e$ satisfying $\e\in(0,\e_0]$, 
\begin{align}\label{2.50}
		&\P \bigg\{\sup_{t\in[0,T]}\|\Y_m^\e(t)-\X_m^\e(t)\|_\H^2>\varpi\bigg\} \leq e^{-\frac{3M}{\e}}.
\end{align}Since $M$ is arbitrary, we obtain the required result \eqref{2.39}.
\end{proof}To finish the proof of our main result (Theorem \ref{thrm2}), we need to verify \eqref{2.4}.
\begin{proof}[Verification of \eqref{2.4}]
By Lemmas \ref{lemma4} and \ref{lemma5}, for any $M>0$, there exists a positive integer $M_0$ satisfying 
\begin{align}\label{2.51}
\P\bigg\{\sup_{t\in[0,T]}\|\Y_{M_0}^\e(t)-\Y^\e(t)\|_\H^2>\frac{\varpi}{3}\bigg\} \leq e^{-\frac{M}{\e}} , \ \text{ for any } \ \e\in(0,1],
\end{align}and 
\begin{align}\label{2.52}
	\P\bigg\{\sup_{t\in[0,T]}\|\X_{M_0}^\e(t)-\X^\e(t)\|_\H^2>\frac{\varpi}{3}\bigg\} \leq e^{-\frac{M}{\e}} , \ \text{ for any } \ \e\in(0,1]. 
\end{align}For such an $M_0$, using Lemma \ref{lemma6}, there exists an $\e_0>0$ such that for every $\e$ satisfying $\e\in(0,\e_0]$, we have
\begin{align}\label{2.53}
	\P\bigg\{\sup_{t\in[0,T]}\|\Y_{M_0}^\e(t)-\X_{M_0}^\e(t)\|_\H^2>\frac{\varpi}{3}\bigg\} \leq e^{-\frac{M}{\e}}.
\end{align}Combining \eqref{2.51}-\eqref{2.53}, for any $\e\in(0,\e_0]$, we have 
\begin{align}\label{2.54}
	\P\bigg\{\sup_{t\in[0,T]}\|\Y^\e(t)-\X^\e(t)\|_\H^2>\varpi\bigg\} \leq e^{-\frac{3M}{\e}}.
\end{align}Since $M$ is arbitrary, we find
\begin{align}\label{2.55}
	\lim_{\e\to0}\e\log\P\bigg\{\sup_{t\in[0,T]}\|\Y^\e(t)-\X^\e(t)\|_\H^2>\varpi\bigg\} =-\infty,
\end{align}which is the required result \eqref{2.4}. Hence Theorem \ref{thrm2} holds using the exponential equivalence result of LDP (cf. Theorem 4.2.13, \cite{ADOZ}).
\end{proof}
\section{Applications}\label{Sec3}\setcounter{equation}{0}
The results obtained in this paper is applicable to a large class of SPDEs with fully local monotone coefficients. One should note that all the models considered in the works (cf. \cite{HBEH,ICAM1,WL4,WLMR1,WLMR2,SLWLYX,EM2,CPMR,MRSSTZ}, etc.) can be covered by our framework, including 2D Navier-Stokes equations, fast-diffusion equations,  porous media equations, $p$-Laplacian equations, Allen-Cahn equations, Burgers equations, 2D Boussinesq system, 3D Leray-$\alpha$ model, 2D Boussinesq model for the Benard convection, 2D magneto-hydrodynamic equations, 2D magnetic B\'enard equations, some shell models of turbulence (Sabra, Goy, Dyadic), power law fluids, the Ladyzhenskaya model, 3D tamed Navier-Stokes equations and the Kuramoto-Sivashinsy equations. In this section, we discuss some important examples which can be covered by the formulation of this paper. The well-posedness  of the following SPDEs is discussed in different works but the formulation provided in \cite{MRSSTZ} cover all such cases. Thus, one can obtain the global solvability results from the work \cite{MRSSTZ} (see Section 4).

\subsection{Cahn-Hilliard equation} \label{ss3.1} The well-known Chan-Hilliard equations were framed in \cite{JWCJEH}, which describe a phase separation in a binary alloy. This model is one of the fundamental equations in material science. This model can be read as 
\begin{equation}\label{3.1}
\left\{	\begin{aligned}
		\partial_tu(t)&=-\Delta^2u(t)+\Delta\varphi(u(t)), \ \ \text{ in } \ \ (0,T)\times\mathcal{O},\\
		\nabla u(t)\cdot\nu &=\nabla (\Delta u(t))\cdot\nu =0,\ \ \text{ on } \ \  [0,T)\times\partial \1,\\
		u(0)&=u_0,\ \text{ in }\ \mathcal{O},
	\end{aligned}
\right.
\end{equation}where $u:[0,T]\times \1\to \R$ denotes the scaled concentration, $\1$ is a bounded domain in $\R^n$ for $n=1,2,3,$ with smooth boundary $\partial\1$ and $\nu$ is the outward drawn unit normal vector on the boundary $\partial \1$. 
Let us assume that the nonlinear term $\varphi$ satisfies the following conditions:
\begin{enumerate}
	\item $\varphi\in\C^1(\R,\R)$,
	\item there exists a positive constant $C$ and $p\in\left[2,\frac{n+4}{n}\right]$ such that for any $x,y\in\R$, 
	\begin{align*}
		\varphi'(x)\geq -C\  \text{  and  }\  |\varphi(x)| \leq C(1+|x|^p),
	\end{align*}and 
	\begin{align*}
		|\varphi(x)-\varphi(y)|\leq C(1+|x|^{p-1}+|y|^{p-1})|x-y|.
	\end{align*}
\end{enumerate}
Let $\H=\mathbb{L}^2(\1)$ and $\V=\{u\in\H^2(\1):\nabla u\cdot\nu=\nabla(\Delta u)\cdot\nu=0\  \text{ on }\  \partial\1\}$. Then, we have the Gelfand triplet $\V\hookrightarrow\H\hookrightarrow\V^*$ and the embedding $\V\hookrightarrow \H$ is compact. Set \begin{align*}
	\A(u)=-\Delta^2u+\Delta\varphi(u).
\end{align*}Then from \cite{MRSSTZ} (see Example 4.4 or Example 5.2.27, \cite{WLMR2}),  we infer that Hypothesis \ref{hyp1} (H.1)-(H.4) hold. More precisely, the local monotonicity property (H.2) can be stated as  
\begin{align*}
&	\langle \A(u)-\A(v),u-v\rangle\\& \leq -\frac{1}{2}\|u-v\|_\V^2+C \big(1+\rho(u)+\eta(v)\big)\|u-v\|_\H^2,
\end{align*}where $\rho(\cdot)$ and $ \eta(\cdot)$ are such that 
\begin{align*}
	|\rho(u)|+|\eta(u)| \leq C\|u\|_\V^{\frac{n(p-1)}{2}}\|u\|_\H^{\frac{(4-n)(p-1)}{2}}.
\end{align*}  One should note that $\frac{n(p-1)}{2}\leq 2 \Leftrightarrow p\leq \frac{n+4}{n}$. For $d=1,2$, the function $\varphi(\cdot)$ can be chosen as $\varphi(x)=x^3-x$, which is the derivative of the double well potential $F(x)=\frac{1}{4}(x^2-1)^2$.

Under the above assumption on $\varphi(\cdot)$, the stochastic counterpart to the system \eqref{3.1} can be stated as follows:
\begin{equation}\label{3.2}
	\left\{\begin{aligned}
	\d \mathrm{Y}(t)&=\big[\A(\mathrm{Y}(t))+\Delta\varphi(\mathrm{Y}(t))\big]\d t+\B(t,\mathrm{Y}(t))\d \W(t), \ \ t\in (0,T),\\
	\mathrm{Y}(0)&=u_0\in\H,
	\end{aligned}
	\right.
\end{equation}where $\W(\cdot)$ is a $\U$-cylindrical Wiener process on the probability space $(\Omega,\mathscr{F},\{\mathscr{F}_t\}_{t\geq0},\P)$ and $\B(\cdot)$ is Lipschitz from $\H$ to $\L_2(\U,\H)$. The well-posedness of the system \eqref{3.2} is established in \cite{MRSSTZ} (see Theorem 2.6). 

Now, consider the small time process  $\mathrm{Y}(\e \cdot)$ and $\varrho_{u_0}^\e$ be the law of $\mathrm{Y}(\e \cdot)$ on $\C([0,T];\H)$. Then using Theorem \ref{thrm2}, we obtain the small time LDP for the problem  \eqref{3.2}.
\begin{theorem}\label{thrm3}
Assume that Hypotheses \ref{hyp1} and \ref{hyp2} hold, and the embedding $\V\hookrightarrow \H$ is compact. Then, \eqref{3.2} has a unique solution $\mathrm{Y}(\cdot)$ for the initial data $u_0\in\H$ and $\varrho_{u_0 }^{\e}$ satisfies the LDP with the rate function $\I(\cdot)$ given by \eqref{1.12}.  
\end{theorem}

\subsection{Stochastic 2D Navier-Stokes equations} \label{ss3.2}
The classical Navier-Stokes equations is a very important model in fluid mechanics which describe the time evolution of fluids. Here we discuss only for two dimensional case and it can be expresses as follows: Let $\1\subset\R^2$ be any bounded domain with smooth boundary $\partial\1$. Consider 
\begin{equation}\label{3.3}
	\left\{	\begin{aligned}
		\partial_t\u(t)&=\nu\Delta\u(t)-(\u(t)\cdot\nabla)\u(t) -\nabla p +\f(t),\ \text{ in }\ (0,T)\times\mathcal{O},\\
	\nabla\cdot\u(t)&=0,\ \text{ in } \ (0,T)\times\1,\\	
	\u&=0, \ \text{ on }\ [0,T)\times\partial \1,  \\	\u(0)&=\u_0,\ \text{ in }\ \mathcal{O},
	\end{aligned}
	\right.
\end{equation}
where $\u(t)=\big(u_1(t,x),u_2(t,x)\big)$ denotes the velocity of the fluid, $p$ is the pressure, $\nu$ is the kinematic viscosity and $f$ is the external forcing of the fluid, and $\u\cdot\nabla=\sum\limits_{i=1}^{2}u_i\partial_i$.

Define \begin{align*}
	\V :=\bigg\{\v\in \mathbb{W}_0^{1,2}(\1;\R^2): \nabla\cdot \v=0, \ \|\v\|_\V^2:=\int_\1|\nabla\v(x)|^2\d x\bigg\}, 
\end{align*} and $\H$ is the closure of the above space $\V$ in the norm $\|\v\|_\H^2 :=\displaystyle \int_\1 |\v(x)|^2\d x$.

Let us define the Stokes operator $\A:\mathbb{W}^{2,2}(\1;\R^2)\cap\V\to\H$ as 
\begin{align*}
	\A(\u)=-\mathrm{P}_\H\Delta\u, \ \text{ for all } \ \u \in \mathbb{W}^{2,2}(\1;\R^2)\cap\V, 
\end{align*}
where $\mathrm{P}_\H$ is the Helmholtz-Leray projection from $\mathbb{L}^2(\1;\R^2)$ to $\H$, and the nonlinear operator 
\begin{align*}
	\mathcal{B}(\u,\v)=\mathrm{P}_\H\big[(\u\cdot\nabla)\v\big], \	\mathcal{B}(\u):=	\mathcal{B}(\u,\u).
\end{align*}
Then, we can write the stochastic counterpart to the system \eqref{3.3}  as
\begin{equation}\label{3.4}
	\left\{	\begin{aligned}
		\d \Y(t)&=-\big[\nu\A(\Y(t)) +\mathcal{B}(\Y(t))-\mathrm{P}_{\H}\f(t)\big]\d t +\B(t,\Y(t))\d\W(t),\ t\in(0,T),\\
				\Y(0)&=\x\in\H,
	\end{aligned}
	\right.
\end{equation} where $\W(\cdot)$ is a $\U$-cylindrical Wiener process on the probability space $(\Omega,\mathscr{F},\{\mathscr{F}_t\}_{t\geq0},\P)$. From \cite{WLMR1} (see Example 3.3), we know that the coefficients of stochastic 2D Navier-Stokes equations satisfy Hypothesis \ref{hyp1} (H.1), (H.4) and \ref{hyp2} (H.2)$^*$.   The well-posedness of the system \eqref{3.4} has been obtained in Theorem 2.6, \cite{SSSPS}. Let $\varrho^\e$ be the law of $\Y(\e\cdot)$ on $\C([0,T];\H)$. Then using Theorem \ref{thrm2}, we obtain the small time LDP for the stochastic 2D Navier-Stokes equations \eqref{3.4}.
 
 \begin{theorem}\label{thrm4}
 	Assume that Hypotheses \ref{hyp1} and \ref{hyp2} hold. Then, \eqref{3.4} has a unique solution $\Y(\cdot)$ for the initial data $\u_0\in\H$ and $\varrho_{\u_0}^\e$ satisfies the LDP with the rate function $\I(\cdot)$ given in \eqref{1.12}. 
 \end{theorem}
\begin{remark}
The small time LDP for the stochastic Navier-Stokes equations have been obtained in \cite{TGXTSZ}. We provide it here since our formulation also covers the system \eqref{3.4}.
\end{remark}

\subsection{Quasilinear SPDEs}\label{ss3.3} Let $\1$ be a bounded domain in $\R^n$ with the smooth boundary $\partial\1$. Let us consider the following quasilinear partial differential equations:
\begin{align}\label{3.5}
	\partial_t u(t,x)=\nabla\cdot \a(t,x,u(t,x),\nabla u(t,x))-\a_0(t,x,u(t,x),\nabla u(t,x)), \ \text{ in } \ (0,T)\times \1,
\end{align}
with the zero Dirichlet boundary conditions (the case of other boundary conditions can be handled in a similar way), where $u:[0,T]\times \1\to\R$, the vector $\nabla u(t,x)=\big(\partial_ju(t,x)\big)_{j=1}^n$ is the gradient of $u(\cdot)$ with respect to the spatial variable $x$ and $\a=(\a_1,\a_2,\ldots,\a_n)$ is a vector with $\a_j:[0,T]\times\1\times\R\times\R^n\to\R^n$ for each $j=0,1,\ldots,n$.

\begin{hypothesis}\label{hyp3}
	Let us assume that $\a_j,\ j=0,1,\ldots,n$ satisfies the following conditions: There exists a constant $\beta >1$ if $n=1,2$ and $\beta\geq \frac{2n}{n+2}$ if $n\geq 3$, such that  the following assumptions hold:
\begin{itemize}
	\item[(A.1)] \textsf{$\a_j$ satisfies the Carath\'eodory conditions:}  for a.e. fixed $(t,x)\in[0,T]\times\1$, $\a_j(t,x,u,z)$ is continuous in $(u,z)\in\R\times\R^n$, for each fixed $(u,z)\in\R\times\R^n$, $\a_j(t,x,u,z)$ is measurable with respect to $(t,x)\in[0,T]\times\1$.
	\item[(A.2)] There exist constants $c_1,c_2\geq0$ and a function $f_1\in \L^\frac{\beta}{\beta-1}((0,T)\times \1;\R^+)$ such that for a.e. $(t,x)\in[0,T]\times\1$ and all $(u,z)\in\R\times\R^n$, $j=1,\ldots,n$,
	\begin{align}\label{3.6}
		|\a_j(t,x,u,z)|\leq c_1|z|^{\beta-1}+c_2|u|^\frac{(\beta-1)(n+2)}{n}+f_1(t,x).
	\end{align}
\item[(A.3)] There exist constants $c_3,c_4\geq0$ and a function $f_2\in\L^1((0,T)\times\1;\R^+)$ such that for a.e. $(t,x)\in[0,T]\times\1$ and all $(u,z)\in \R\times\R^n$,
\begin{align}\label{3.7}
	\sum_{j=1}^n \a_j(t,x,u,z)z_j+\a_0(t,x,u,z)u \geq c_3|z|^\beta-c_4|u|^2-f_2(t,x).
\end{align}
\item[(A.4)] For a.e. $(t,x)\in [0,T]\times\1$ and all $u\in\R$ and $z^1,z^2\in\R^n$, we have 
\begin{align}\label{3.8}
	\sum_{j=1}^n \big[\a_j(t,x,u,z^1)-\a_j(t,x,u,z^2)\big](z^1_j-z^2_j)>0.
\end{align}For a.e. $(t,x)\in[0,T]\times\1$, $z\in\R^n$ and for any $R>0$,
\begin{align}\label{3.9}
	\lim_{|z|\to\infty} \frac{\sup\limits_{|\u|\leq R}\sum\limits_{j=1}^{n}\a_j(t,x,u,z)z_j}{|z|+|z|^{\beta-1}}=\infty.
\end{align}
\end{itemize}
	\end{hypothesis}

Let $\H:=\mathbb{L}^2(\1)$ and $\V:=\mathbb{W}_0^{1,\beta}(\1)$. Using the Sobolev embedding we have the Gelfand triplet $\V\hookrightarrow\H\hookrightarrow\V^*$, with the compact embedding $\V\hookrightarrow\H$. 

For $u,v\in\V$, the operator $\A(\cdot)$ is defined as follows: 
\begin{align}\label{3.10}
	\langle \A(t,u),v\rangle &=-\int_\1\bigg\{\sum_{j=1}^{n}\a_j(t,x,u(x),\nabla u(x))\partial_ju(x)+\a_0(t,x,u(x),\nabla u(x))v(x)\bigg\}\d x.
\end{align} Using the Gagliardo-Nirenberg inequality for $1\leq p\leq \infty$, we have 
\begin{align}\label{3.11}
	\|u\|_{\mathbb{L}^p}\leq C\|\nabla u\|_{\mathbb{L}^\beta}^\delta\|u\|_{\mathbb{L}^2}^{1-\delta}, \text{ where } \delta\in[0,1] \text{ and } \frac{1}{p}=\left(\frac{1}{\beta}-\frac{1}{n}\right)\delta+\frac{1-\delta}{2}.
\end{align}
Then, form (A.2) and \eqref{3.11}, we obtain that   $\A(\cdot)$ is a measurable  mapping from $[0,T]\times\V$ to $\V^*$. Also, we have 
\begin{align}\label{3.12}
	\|\A(t,u)\|_{\V^*} \leq c_1\|u\|_\V^\beta+c\|u\|_\V^\beta\|u\|_\H^\frac{2\beta}{n}+F(t),
\end{align}where the function $ F(t)=\displaystyle\int_\1 f_1(t,x)^{\frac{\beta}{\beta-1}}\d x$ is in the space $\L^1(0,T;\R^+)$, which implies Hypothesis \ref{hyp1} (H.4).
 
 Using Hypothesis \ref{hyp3} (A.1) and (A.2), we obtain the hemicontinuity property (H.1). Combining (A.3) and \eqref{3.11}, we obtain the coercivity property (H.3). Again, by (A.1), (A.2) and (A.4), one can show that the operator $\A(\cdot)$ is pseudo-monotone for a.e. $t\in[0,T]$ (see Theorems 10.65 and 10.63, \cite{MRRCR} or Theorem 2.8, \cite{JLL}). By Corollary 2.7, \cite{MRSSTZ}, we obtain the existence of a \textsf{probabilistically weak solution} to the corresponding stochastic partial differential equations.
 
 A typical example of the formulation \eqref{3.5} is the $p$-Laplacian for $p\geq2$, 
 \begin{align}\label{3.13}
 	\partial_tu=\nabla\cdot\big(|\nabla u|^{p-2}\nabla u\big)-c|u|^{p-2}u,
 \end{align}where $c>0$.  Fix $\beta=p$, and one can see that Hypothesis \ref{hyp3} remains valid. 
 
 Since our aim is to establish LDP, therefore we need the existence of \textsf{probabilistically strong solutions}. For that we have to modify Hypothesis \ref{hyp3} (A.4), by the following condition: 
 \begin{hypothesis}\label{hyp4}
 \begin{itemize}
 	\item[(A.4)$^*$]  For $\beta\geq n$, $\gamma\in \left[0,\beta\big(1+\frac{2}{n}\big)-2\right]$ and $f_3\in\L^1(0,T;\R^+)$,  there exists a constant $c>0$ such that for a.e. $(t,x)\in[0,T]\times\1$ and all $u^1,u^2\in\R$ and $z^1,z^2\in\R^n$, 
\begin{align}\label{3.14}\nonumber
	&\sum_{j=1}^n \big[\a_j(t,x,u^1,z^1)-\a_j(t,x,u^2,z^2)\big](z^1_j-z^2_j)\\&+\big[\a_0(t,x,u^1,z^1)-\a_0(t,x,u^2,z^2)\big](u^1-u^2)>-c\big(f_3(t)+|u^1|^\gamma+|u^2|^\gamma \big)|u^1-u^2|^2.
\end{align}
\end{itemize}
\end{hypothesis}
Under the above assumption, we find   
\begin{align}\label{3.15}
	\langle \A(t,u)-\A(t,v),u-v\rangle \leq \big(f_3(t)+\rho(u)+\eta(v)\big)\|u-v\|_\H^2,
\end{align}where $\rho(\cdot)$ and $\eta(\cdot)$ are such that 
\begin{align*}
	|\rho(u)|+|\eta(u)|\leq C\|u\|_\V^{\gamma\delta}\|u\|_\H^{\gamma(1-\delta)},
\end{align*}
with $\delta=\frac{\beta n}{\beta n+2\beta-2n}$. Thus, the local monotonicity property (H.2) is verified.  Then, there exists a unique \textsf{probabilistically strong solution} (see Theorem 2.6, \cite{MRSSTZ}).
 
 Let $\varrho^\e$ be the law of small time process $\mathrm{Y}(\e\cdot)$ on $\C([0,T];\H)$. Then, using Theorem \ref{thrm2}, we establish the small time LDP for the stochastic counterpart to the system \eqref{3.5} with Hypothesis \ref{hyp2} (H.6) for the noise  coefficients $\B(\cdot,\cdot)$. Thus, we have  the following theorem:
 
  \begin{theorem}\label{thrm5}
 Let  Hypotheses \ref{hyp1} and  \ref{hyp2} be satisfied. Then there exists a unique solution $\mathrm{Y} (\cdot)$ of the stochastic counterpart of the system \eqref{3.5} for the initial data $u_0\in\H$ and $\varrho_{u_0}^\e$ satisfies the LDP with the rate function $\I(\cdot)$ given in \eqref{1.12}. 
 \end{theorem}

\subsection{Convection-diffusion equation}\label{ss3.4} The convection-diffusion equation describes the physical phenomena where  particles, energy or any other physical quantities injected into a physical system due to two processes: \textsf{diffusion} and  \textsf{convection}. It has a variety of applications in fluid dynamics, mass and heat transfer, etc. We consider the following system on the $n$-dimensional torus $\mathbb{T}^n$:
\begin{equation}\label{3.16}
	\left\{	\begin{aligned}
		\partial_t u(t)&=\nabla\cdot\big(\a(u)\nabla u+\b(u)\big), \ \text{ on } \ (0,T)\times \mathbb{T}^n,\\
		u(0)&=u_0\in\H,
	\end{aligned}
	\right.
\end{equation}where $u:[0,T]\times \mathbb{T}^n\to\R$ is the flux function, $\b=(b_1,\ldots,b_n):\R\to\R^n$, and the diffusion matrix $\a=(a_{ij})_{i,j=1}^{n}:\R^n\to\mathbb{M}_{n\times n}$ here $\mathbb{M}_{n\times n}$ is the set of all real $n\times n$ matrices. The coefficients $\a(\cdot)$ and $\b(\cdot)$ satisfy the following conditions:
\begin{enumerate}
	\item Both $\a(\cdot)$ and $\b(\cdot)$ are continuous;
	\item $\b(\cdot)$ has linear growth;
	\item $\a(\cdot)$ is bounded and uniformly positive definite, that is, there exist positive constants $\delta$ and $C$ such that for any $u\in\R$ and $\z\in\R^n$, 
	\begin{align}\label{3.17}
		\delta|\z|^2\leq \langle \a(u)\z,\z\rangle\leq C|\z|^2.
	\end{align} 
\end{enumerate}

Let $\H:=\mathbb{L}^2(\mathbb{T}^n)$ and $\V:=\mathbb{W}^{1,2}(\mathbb{T}^n)$.  Then, we have the Gelfand triplet $\V\hookrightarrow\H\hookrightarrow\V^*$ with the compact embedding $\V\hookrightarrow\H$. For any $u,v\in\V$, define the operator $\A(\cdot)$ as 
\begin{align}\label{3.18}
	\langle\A(u),v\rangle =-\int_{\mathbb{T}^n}(\a(u(x))\nabla u(x)+\b(u(x)),\nabla v(x))\d x.
\end{align}
Under the above conditions, the coefficients in the system  \eqref{3.16} satisfies Hypothesis \ref{hyp3}, which implies Hypothesis  \ref{hyp1} holds except local monotonicity condition (H.2). Also, the operator $\A(\cdot)$ defined in \eqref{3.18} is a pseudo-monotone operator (cf. Example 4.2, \cite{MRSSTZ}). 

For the stochastic counterpart of \eqref{3.16}, we assume that the diffusion coefficient is globally Lipschitz in $\H$.  By Corollary 2.7, \cite{MRSSTZ}, we deduce the existence of \textsf{probabilistically weak solutions}, for the corresponding stochastic equation of the problem \eqref{3.16}. With an additional assumption, that is, the coefficients $\a(\cdot)$ and $\b(\cdot)$ are Lipschitz, we obtain the existence and uniquenss of \textsf{probabilistically strong solution} for the corresponding stochastic equation to the problem \eqref{3.16} (cf. Example 4.2, \cite{MRSSTZ} and Theorem 3.1, \cite{MHTZ}).

Let $\varrho^\e$ be the law of small time process $\mathrm{Y}(\e\cdot)$ on $\C([0,T];\H)$. Then, using Theorem \ref{thrm2}, we deduce the small time LDP for the corresponding stochastic equation to \eqref{3.16} with Hypothesis \ref{hyp2} (H.6) for the noise  coefficient $\B(\cdot,\cdot)$. Thus, we have the following theorem:
\begin{theorem}\label{thrm6}
	Under assumptions \ref{hyp1} and \ref{hyp2}, there exists a unique solution $\mathrm{Y}(\cdot)$ of the corresponding stochastic equation to \eqref{3.13} for the initial data $u_0\in\H$ and $\varrho_{u_0}^\e$ satisfies the LDP with the rate function $\I(\cdot)$ given in \eqref{1.12}. 
\end{theorem}
\subsection{2D Liquid crystal model}\label{ss3.5} The elementary form of the hydrodynamics of liquid crystals is a simplified version of Ericksen-Leslie system with Ginzburg-Landau approximation, which is discussed in \cite{FLCL}. We consider the following model in the bounded domain $\1$ of $\R^2$ with smooth boundary $\partial\1$, 
\begin{equation}\label{3.19}
	\left\{	\begin{aligned}
		\partial_t \u(t)&=\Delta\u(t)-(\u(t)\cdot\nabla)\u(t)-\nabla p-\nabla \cdot(\nabla \n(t)\otimes\nabla\n(t)),    \text{ in }  (0,T)\times\1,\\
		\nabla\cdot\u(t)&=0,\ \text{ on } \ (0,T)\times\partial\1,\\
		\partial_t\n(t)&=\Delta\n(t)-(\u(t)\cdot\nabla)\n(t)-\Phi(\n(t)), \ \text{ in } \ (0,T)\times\1,\\
		\u&=0, \ \text{ on } [0,T)\times\partial\1,\\ \frac{\partial \n}{\partial \nu }&=0,\text{  on  } [0,T)\times\partial \1,\\
		\big(\u(0),&\n(0)\big)=\big(\u_0,\n_0\big), \ \text{ in } \ \1,
	\end{aligned}
	\right.
\end{equation}where $\u:[0,T]\times\1\to\R^2$ is the velocity, $p:[0,T]\times\1\to\R$ is the pressure, $\n:[0,T]\times\1\to\R^3$ is the director field of the liquid crystal molecules and $\nu$ is the outward unit normal vector on the boundary $\partial\1$. The symbol $\nabla \n\otimes\nabla\n$ denotes a $2\times2$ matrix whose entries are given by 
\begin{align*}
	(\nabla \n\otimes\nabla\n)_{i,j}=\sum_{k=1}^3(\partial_in_k)(\partial_jn_k).
\end{align*}Furthermore, we assume that $\Phi:\R^2\to\R^3$ satisfies the following conditions: there exists a $k$ degree polynomial $\psi:[0,\infty)\to\R$ for some $k\in\N$, such that 
\begin{align*}
	\Phi(\n)=\psi(|\n|^2)\n=\bigg(\sum_{i=0}^ka_i|\n|^{2i}\bigg)\n,
\end{align*}where $a_i\in \R$ for $i=1,\ldots,k-1$ and $a_k>0$. 

Let $\V=\{\u\in\H^1(\1)^2:\nabla\cdot\u=0,\ \u\big|_{\partial\mathcal{O}}=0\}$ and $\H$ be closure of $\V$ under the $\mathbb{L}^2$-norm $\|\u\|_\H^2:=\displaystyle\int_\1|\u(x)|^2\d x$. Set
\begin{align*}
	\mathcal{H}:=\H\times\H^1(\1)^3 \ \text{ and } \ \mathcal{V} :=\V\times \left\{\n\in\H^2(\1)^3:\frac{\partial \n}{\partial \nu }=0\right\},
\end{align*}with the norms in $\mathcal{H}$ and $\mathcal{V}$ denoted  by 
\begin{align*}
	\|\Y\|_\mathcal{H}^2:=\|\u\|_\H^2+\|\n\|_{\H^1}^2,\ \text{ and } \  	\|\Y\|_\mathcal{V}^2:=\|\u\|_\V^2+\|\n\|_{\H^2}^2, 
\end{align*}for $\Y=(\u,\n)$, respectively and $\langle\cdot,\cdot\rangle_{\mathcal{V}^*\times\mathcal{V}}$ denotes the duality pairing between $\mathcal{V}$ and $\mathcal{V}^*$. Thus, we have a Gelfand triplet $\mathcal{V}\hookrightarrow\mathcal{H}\hookrightarrow\mathcal{V}^*$ with the compact embedding $\mathcal{V}\hookrightarrow\mathcal{H}$. 

One should note that 
\begin{align}\label{3.20}
	\nabla\cdot(\nabla \n\otimes\nabla\n)=\frac{1}{2}\nabla(|\nabla\n|^2)+\nabla\n\cdot\Delta\n.
\end{align}Let us define the Helmholtz-Leray projection $\mathrm{P}_\H$ from $\mathbb{L}^2(\1)^2$ to $\H$, and define
\begin{equation}\label{3.21}
	\A(\Y):=\left( \begin{aligned}\mathrm{P}_\H\big[&\Delta \u-(\u\cdot\nabla)\u-\nabla\n\cdot\Delta\n\big]\\ &\Delta\n-(\u\cdot\nabla)\u-\psi(\n)
		\end{aligned}\right).
\end{equation}The operator $\A:\mathcal{V}\to\mathcal{V}^*$ satisfies Hypothesis \ref{hyp1}  (H.1)-(H.4), since (cf. Example 4.5, \cite{MRSSTZ})
\begin{align*}
	\|\A(\Y)\|_{\mathcal{V}^*}^2\leq C\big(1+\|\Y\|_{\mathcal{H}}^{4k+2}\big)\|\Y\|_{\mathcal{V}^*}^2, \ \text{ and } \ 	\langle \A(\Y),\Y\rangle_{\mathcal{V}^*\times\mathcal{V}}  \leq -\|\Y\|_\mathcal{V}^2+C\|\Y\|_\mathcal{H}^2.
\end{align*}Also, we have
\begin{align*}
	&	\langle \A(\Y_1)-\A(\Y_2),\Y_1-\Y_2\rangle_{\mathcal{V}^*\times\mathcal{V}} \leq -\frac{1}{2}\|\Y_1-\Y_2\|_\mathcal{V}^2+C\big(1+\rho(\Y_1)+\eta(\Y_2)\big)\|\Y_1-\Y_2\|_\mathcal{H}^2,
\end{align*}where $\rho(\cdot)$ and $\eta(\cdot)$ are such that 
\begin{align*}
	|\rho(\Y_1)|\leq \|\Y_1\|_\mathcal{H}^{4k}, \ \ \text{ and } \ \ |\eta(\Y_2)|\leq \|\Y_2\|_\mathcal{H}^{4k}+\|\Y_2\|_\mathcal{H}^2\|\Y_2\|_\mathcal{V}^2.
\end{align*}
Thus, the stochastic counterpart to the system \eqref{3.19} has a unique \textsf{probabilistically strong solution} under Hypothesis \ref{hyp1} (for more details see Example 4.5, \cite{MRSSTZ}).
\begin{remark}[Remark 4.6, \cite{MRSSTZ}]\begin{enumerate}\item 
	The authors in the work \cite{ZBEHPAR} considered the stochastic counterpart to the system \eqref{3.19} with the noise in the equation for $\u$ depending only on $\u$  in the It\^o sense, with linear multiplicative noise only depending on $\n$ in the Stratonovich sense in the equation for $\n$. Using Theorem 2.6, \cite{MRSSTZ}, we obtain the existence and uniqueness of a \textsf{probabilistically strong solution} to stochastic 2D liquid crystal equations perturbed by multiplicative noise which can depend on both $\u$ and $\n$.
	\item  If we assume $n:\1\to\R$ and $\Phi:\R\to\R$ are scalar functions appearing in the system \eqref{3.19}, then the corresponding system is called \emph{Allen-Cahn-Navier-Stokes equations (ACNSEs)}, which describe the motion of mixture of two incompressible viscous fluids. For more details one can see \cite{CLJS,TTM}, etc. and references therein. ACNSEs is also related to the magneto-hydrodynamics equations (MHDEs), that is the Navier-Stokes equations coupled with the Maxwell equations. In particular, for two dimensions and the nonlinear term $\Phi(n)=0$, the resultant model is equivalent to MHDEs (see \cite{XXLZCL}). The well-posedness of both models, that is, ACNSEs and MHDEs are covered by the formulation in \cite{MRSSTZ}, and hence it falls into the framework discussed in this paper. 
	\end{enumerate}
\end{remark} 

Let $\varrho^\e$ be the law of the small time process $\Y(\e\cdot)$ on $\C([0,T];\mathcal{H})$. Then, using our main theorem (Theorem \ref{thrm2}), we obtain the small time LDP for the corresponding stochastic equation to \eqref{3.19} with Hypothesis \ref{hyp2} (H.6) for the noise  coefficients $\B(\cdot,\cdot)$. Thus, we have the following theorem:
\begin{theorem}\label{thrm7}
	Assume that Hypotheses \ref{1.1} and \ref{hyp2} hold and the embedding $\mathcal{V}\hookrightarrow\mathcal{H}$ is compact. Then, there exists a unique solution $\Y(\cdot)$ of the corresponding stochastic equation to \eqref{3.19} for the initial data $(\u_0,\n_0)\in\mathcal{H}$ and $\varrho_{\u_0,\n_0}^\e$ satisfy the LDP with the rate function $\I(\cdot)$ given in \eqref{1.12}. 
\end{theorem}

\medskip\noindent
\textbf{Acknowledgments:} The first author would like to thank Ministry of Education, Government of India - MHRD for financial assistance. M. T. Mohan would  like to thank the Department of Science and Technology (DST), India for Innovation in Science Pursuit for Inspired Research (INSPIRE) Faculty Award (IFA17-MA110).

\medskip\noindent	{\bf  Declarations:} 

\noindent 	{\bf  Ethical Approval:}   Not applicable 

\noindent  {\bf   Competing interests: } The authors declare no competing interests. 

\noindent 	{\bf   Authors' contributions: } All authors have contributed equally. 

\noindent 	{\bf   Funding: } DST, India, IFA17-MA110 (M. T. Mohan).

\noindent 	{\bf   Availability of data and materials: } Not applicable.

\end{document}